\documentclass[12pt,reqno]{amsart}

\usepackage{amsfonts,amsthm,mathrsfs,amssymb,amscd,enumerate,url,tikz-cd}

\input{xypic}
\xyoption{all}

\usepackage{xcolor}
\colorlet{mdtRed}{red!50!black}
\definecolor{dblue}{rgb}{0,0,.6}
\usepackage[colorlinks]{hyperref}
\hypersetup{linkcolor=blue,citecolor=dblue,filecolor=dullmagenta,urlcolor=mdtRed}

\numberwithin{equation}{section}

\newtheorem{theorem}{Theorem}[section]
\newtheorem{proposition}[theorem]{Proposition}
\newtheorem{lemma}[theorem]{Lemma}
\newtheorem{corollary}[theorem]{Corollary}

\theoremstyle{definition}
\newtheorem{remark}[theorem]{Remark}

\newcommand{\C}{\mathbb{C}}

\newcommand{\bq}{\overline q}

\newcommand{\mf}[1]{\mathfrak{#1}}
\newcommand{\ms}[1]{\mathscr{#1}}
\newcommand{\mb}[1]{\mathbb{#1}}
\newcommand{\mc}[1]{\mathcal{#1}}
\renewcommand{\t}[1]{\widetilde{#1}}

\usepackage{marginnote}
\usetikzlibrary{calc}

\begin{document}

\title[Infinitesimal deformations of some Quot schemes]
{Infinitesimal deformations of some Quot schemes} 

\author[I. Biswas]{Indranil Biswas} 
	
\address{Department of Mathematics, Shiv Nadar University, NH91, Tehsil Dadri, Greater Noida, Uttar
Pradesh 201314, India} 

\email{indranil.biswas@snu.edu.in, indranil29@gmail.com} 

\author[C. Gangopadhyay]{Chandranandan Gangopadhyay} 

\address{Department of Mathematics, Indian Institute of Science Education and Research, Pune, 411008, Maharashtra, India.}

\email{chandranandan@iiserpune.ac.in}

\author[R. Sebastian]{Ronnie Sebastian} 

\address{Department of Mathematics, Indian Institute of Technology Bombay, Powai, Mumbai 
400076, Maharashtra, India}

\email{ronnie@math.iitb.ac.in} 



\begin{abstract}
Let $E$ be a vector bundle on a smooth complex projective curve $C$
of genus at least two. Let $\mc Q(E,d)$ be the Quot 
scheme parameterizing the torsion quotients of $E$ of degree $d$. 	
We compute the cohomologies of the tangent bundle $T_{\mc Q(E,d)}$. In particular, the
space of infinitesimal deformations of $\mc Q(E,d)$ is computed. Kempf and
Fantechi computed the space of infinitesimal deformations of $\mc Q({\mathcal O}_C,d)\,=\,
C^{(d)}$ (\cite{Kempf}, \cite{Fantechi}).
\end{abstract}

\maketitle

\tableofcontents

\section{Introduction}

Let $C$ be a smooth projective curve over $\mb C$ of genus $g_C$, with
$g_C\, \geq\,2$. 
Let $E$ be a vector bundle on $C$ of rank $r\, \geqslant\, 1$. 
Fix an integer $d\,\geqslant\, 1$. Let $\mc Q\,:=\,\mc Q(E,d)$ be the 
Quot scheme parameterizing all torsion quotients of 
$E$ of degree $d$. It is known that $\mc Q$ is a 
smooth projective variety of dimension $rd$. This Quot scheme has 
various moduli theoretic interpretations, 
see \cite{BGL}, \cite{BDW}, \cite{BFP},
\cite{HP}, \cite{BR}, \cite{OP},
which have led to extensive studies of it. Our aim here is to compute
the cohomologies of the tangent bundle $T_{\mc Q}$, especially
$H^1(\mc Q,\, T_{\mc Q})$ that parametrizes the infinitesimal deformations
of $\mc Q$.

The group of holomorphic automorphisms $\text{Aut}(\mc Q)$ of $\mc Q$ is a complex Lie group
whose Lie algebra is $H^0(\mc Q,\,T_{\mc Q})$ \cite[Lemma 1.2.6]{Sern}.
It is known that 
$$H^0(\mc Q({\mathcal O}^{\oplus r},d),\,
T_{\mc Q({\mathcal O}^{\oplus r},d)})\,=\, \mf{sl}(r, {\mathbb C})\,=\,
H^0(X,\, {\rm End}({\mathcal O}^{\oplus r}))/{\mathbb C}
$$
for all $r\, \geq\, 2$ \cite{BDH-aut}. From this it follows that the maximal
connected subgroup of $\text{Aut}(\mc Q({\mathcal O}^{\oplus r},d))$ is
${\rm PGL}(r,{\mathbb C})\,=\, \text{Aut}({\mathcal O}^{\oplus r})/{\mathbb C}^*$.
More generally, if either $E$ is semistable or $r\,\geqslant\, 3$,
then
$$
H^0(\mc Q(E,d),\, T_{\mc Q(E,d)})\,=\, H^0(X,\, \text{End}(E))/{\mathbb C}
$$
\cite{G19}, and hence the maximal connected subgroup of $\text{Aut}(\mc Q(E,d))$
is $\text{Aut}(E)/{\mathbb C}^*$.

Regarding the next cohomology $H^1(\mc Q,\,T_{\mc Q})$, first consider the
case of $r\,=\,1$. In this case, the Quot scheme
$\mc Q(E,d)$ is identified with the $d$-th symmetric product 
$C^{(d)}$ of $C$. The infinitesimal deformation space
$H^1(C^{(d)},\, T_{C^{(d)}})$ was computed in \cite{Kempf} under
the assumption that $C$ is 
non-hyperelliptic, and it was computed in \cite{Fantechi} when 
$g\,\geqslant\, 3$ (see also \cite[Remark 2.6]{Fantechi} for the case $g=2$). 

Henceforth, we will always assume that $r\, =\, {\rm rank}(E)\, \geqslant\, 2$.

If $d\,=\,1$, then $\mc Q\cong \mb P(E)$. Consequently,
$H^1(\mc Q(E,1),\,T_{\mc Q(E,1)})$ can be computed easily.

Associated to the vector bundle $E$ there is the Atiyah bundle $At(E)$ on $C$
\cite[Theorem 1]{Atiyah}. The infinitesimal 
deformations of the pair $(C,\, E)$ are parametrized by
$H^1(X,\, At(E))$ \cite[Proposition 4.2]{Chen}. For the natural
homomorphisms $\mc O_C\,\hookrightarrow \,\text{End}(E)
\,\hookrightarrow \, At(E)$, the quotients $\text{End}(E)/\mc O_C$
and $At(E)/\mc O_C$ are vector bundles and will be denoted by $ad(E)$ and $at(E)$ respectively. 
Also, given any vector bundle $V$ on $C$ we can 
construct a natural bundle called the Secant bundle 
${Sec}^d(V)$ on $C^{(d)}$ 
(see \cite[Proposition 1]{Mattuck}, \cite[Section 2]{biswas-laytimi}). 
In particular we have a bundle $Sec^d(at(E))$ on $C^{(d)}$.

Recall that we have the Hilbert-Chow map $\phi:\mc Q\longrightarrow C^{(d)}$ 
(also called Quot-to-Chow morphism by some authors 
\cite[Section 2.4]{Ricolfi}). We refer to 
\cite[Section 2]{GS} for the definition of this map. 

We prove the following (see Theorem \ref{main theorem}):

\begin{theorem}\label{theorem 1}
Let $r\,=\, {\rm rank}(E)\,\geqslant \,2$. Then
\begin{enumerate}
\item $Sec^d(at(E))\,\cong\, \phi_*T_{\mc Q}$ and

\item $R^i\phi_*T_{\mc Q}\,=\,0$ for all $i\,>\,0$.
\end{enumerate}
\end{theorem}

The cohomologies of $Sec^d(at(E))$ can be computed easily. Therefore,
we can compute the spaces $H^i(\mc Q, \,T_{\mc Q})$ using Theorem \ref{theorem 1}.
We prove the following (see Theorem \ref{cor-cohomology of T_Q}):

\begin{theorem}\label{theorem 2}
Let ${\rm rank}(E),\, d\,\geqslant\, 2$. Denote the genus of $C$ by $g_C$.
The following three statements hold:
\begin{enumerate}
\item For all $d-1\, \geqslant\, i\, \geqslant\, 0$,
\begin{align*}
H^i(\mc Q,\,T_{\mc Q})\,=\,H^0(C,\,at(E))\otimes & \bigwedge^ i H^1(C,\,\mc O_{C})\\
& \bigoplus H^1(C,\,at(E))\otimes 
\bigwedge^{i-1}H^{1}(C,\,\mc O_{C})\, .
\end{align*}
In particular,
$$h^i(\mc Q,\, T_{\mc Q})\,=\, {g_C \choose i} \cdot h^0(C,\,at(E))+
{g_C \choose {i-1}} \cdot h^1(C,\, at(E))\,.$$

\item When $i\,=\,d$,
$$H^d(\mc Q, \, T_{\mc Q})\, =\, \bigwedge^{d-1}H^1(C,\,\mc O_C) \otimes h^1(C, \, at(E))\,.$$
In particular,
$$h^d(\mc Q, \, T_{\mc Q})\, =\, {g_C \choose {d-1}} \cdot h^1(C, \, at(E))\,.$$

\item
For all $i\,\geq\, d+1$, $$H^i(\mc Q,\, T_{\mc Q})\,=\,0.$$
\end{enumerate}
\end{theorem}

By Theorem \ref{theorem 2} when $g_C\geq 2$ we get
$$
H^0(\mc Q,\,T_{\mc Q})\,=\,H^0(C,\,at(E))\,=\,H^0(C,\,ad(E))\,.
$$
Recall that $\mc Q$ is a fine moduli space, that is, 
there exists a certain universal quotient on $C\times \mc Q$. 
The kernel of this universal quotient, which is locally free, is denoted
by $\mc A$. We also 
compute $H^1(C\times \mc Q,\,{\ms End}(\mc A))$, which
is the space of all infinitesimal
deformations of $\mc A$. More precisely, the following is proved
(see Corollary \ref{c916}):

\begin{theorem}\label{theorem 3}
Let ${\rm rank}(E),\,g_C,\,d\,\geqslant\, 2$. Then we have
$$H^1(C\times \mc Q,\,{\ms End}(\mc A))
\,=\,H^1(C\times \mc Q,\,\mc O_{C\times \mc Q})
\,=H^1(C,\,\mc O_C)\oplus H^1(\mc Q,\,\mc O_{\mc Q})\,.$$
\end{theorem}

It can be seen using Corollary \ref{cor-direct image of structure sheaf} 
that $H^1(\mc Q,\mc O_{\mc Q})\,=\,H^1(C,\mc O_{C})$.

In \cite{G18} it was shown that when $E\,\cong \,\mc O^n_C$ 
for some $n\,\geqslant \,1$, then $\mc A$ is slope stable with 
respect to some natural polarizations of $C\times \mc Q$. 
By \cite[Corollary 4.5.2]{HL}, the space $H^1(C\times \mc Q,\,{\ms End}(\mc A))$ is the tangent 
space at $[\mc A]$ of the Moduli space $\mc M$ of sheaves on 
$C\times \mc Q$ with the same Hilbert polynomial 
(with respect to some fixed polarization on $C\times \mc Q$) as $\mc A$. 
Moreover, the differential of the determinant map $\mc M\,\longrightarrow\, {\rm Pic}(C\times \mc Q)$ at the point $[\mc A]$ is given by 
the trace map 
$$H^1(C\times \mc Q,\,{\ms End}(\mc A))\xrightarrow{H^1(tr)} H^1(C\times \mc Q, \mc O_{C\times \mc Q})$$ (see \cite[Theorem 4.5.3]{HL}).
 Note that this map is onto, since the composition of the maps $\mc O_{C\times \mc Q}\to {\ms End}(\mc A) \xrightarrow{tr} \mc O_{C\times \mc Q}$ is an isomorphism. 
Therefore, Theorem \ref{theorem 3} implies 
that the determinant map $\mc M\,\longrightarrow\, {\rm Pic}(C\times \mc Q)$
induces an isomorphism at the level of tangent spaces 
at the point $[\mc A]$.

We briefly describe how the paper is organized. In Sections 2 to 5
we prove several preliminary results which we shall need later. 
In these sections we show that the relative adjoint Atiyah sequence 
associated to a vector bundle $V$ on $C\times X$ (see \eqref{eqn rel ad-Atiyah}),
restricted to $c\times X$ can be obtained in three different ways.
In Section
6 we recall the construction of the space $S_d$ from \cite{GS} and prove
a result relating to its canonical divisor. In Section 
7 we prove some results related to projective bundles. 
In Section 8, the results in Section 7 and Section 5
are used in computing cohomologies of some sheaves on $S_d$. These 
cohomology computations are used in Section 9 to compute higher 
direct images, for the Hilbert-Chow map, of some natural vector bundles 
on $\mc Q$. These computations are then used to prove the main
results. 

\section{Some tangent bundle sequences}\label{section-tangent bundle of projective bundle}

The base field is $\mathbb C$.
As before, $C$ is a smooth projective curve. Let $X$ be a smooth projective variety and 
$V\, \longrightarrow\, C\times X$ a vector bundle.
Let $$\pi\,:\,\mb{P}(V)\, \longrightarrow\, C\times X$$ be the projective bundle
parametrizing the hyperplanes in the fibers of $V$.
Denoting the natural projections of $C\times X$ to $C$ and $X$ by $q_C$ and $q_X$ respectively,
define
\begin{equation}\label{e1}
\pi_C\,:=\, q_C\circ\pi\, :\, \mb{P}(V)\, \longrightarrow\, C\ \ \text{ and }\ \
\pi_X\,:=\, q_X\circ\pi\, :\, \mb{P}(V)\, \longrightarrow\, X\, .
\end{equation}
Let $\mc O(1)\, :=\, {\mc O}_{\mb{P}(V)}(1)
\, \longrightarrow\, \mb{P}(V)$ be the tautological line bundle, and let 
$\Omega_{\pi}\, \longrightarrow\, \mb{P}(V)$ be the relative cotangent bundle
for the morphism $\pi$. We have the relative Euler sequence on $\mb P(V)$:
\begin{equation}\label{eqn-relative euler}
0 \, \longrightarrow\, \Omega_{\pi}(1)\, :=\,
\Omega_{\pi}\otimes \mc O(1)\, \longrightarrow\, \pi^*V
\, \stackrel{q'}{\longrightarrow}\, \mc O(1)\, \longrightarrow\, 0\,.
\end{equation}
Define
$$p_{1,2X}\,:\,C\times \mb P(V)\, \longrightarrow\, C\times X,\ \ (c,\,v)\, \longmapsto
\, (c,\, \pi_X(v)),
$$
where $\pi_X$ is the map in \eqref{e1}.
Let
\begin{equation}\label{e3}
i\,:\,\mb P(V)\,\hookrightarrow\, C\times \mb P(V),\ \
v\,\longmapsto\, (\pi_C(v),\, v)
\end{equation} 
be the closed embedding, where $\pi_C$ is the map in \eqref{e1}. The composition of maps 
$$\mb P(V)\,\xrightarrow{\ i\ }\, C\times \mb P(V)\,\xrightarrow{\,p_{1,2X}}\, C\times X$$
evidently coincides with $\pi$. Let
$$
q\, :\, p_{1,2X}^*V\, \longrightarrow\,i_*\mc O(1)
$$
be the surjective homomorphism given by the following composition of homomorphism
of sheaves
$$p_{1,2X}^*V\, \longrightarrow\, i_*i^*p_{1,2X}^*V\,\cong\, i_*\pi^*V\,\stackrel{i_*q'}{\longrightarrow}
\, i_*\mc O(1)\,\longrightarrow \, 0$$
on $C\times \mb P(V)$, where $q'$ is the projection in \eqref{eqn-relative euler}. Denote
$$
\mc V\, :=\, \text{kernel}(q)\, \subset\, p_{1,2X}^*V\, ;
$$
so we have the short exact sequence
\begin{equation}\label{eqn-universal exact sequence on Q(E,1)}
0\, \longrightarrow\, \mc V\, \longrightarrow\, p_{1,2X}^*V \,\xrightarrow{\ q\ }\,
 i_*\mc O(1) \, \longrightarrow\, 0\,.
\end{equation}
Restricting \eqref{eqn-universal exact sequence on Q(E,1)} to $i(\mb P(V))\,=\, \mb P(V)$ 
we get a right exact sequence
$$i^*\mc V\, \longrightarrow\, i^*p_{1,2X}^*V\,\cong\, \pi^*V
\, \stackrel{q'}{\longrightarrow}\, \mc O(1)\, \longrightarrow\, 0\, ,$$
where $i$ and $q'$ are the maps in \eqref{e3} and \eqref{eqn-relative euler} respectively.
Therefore we get a surjective homomorphism from $i^*\mc V$ to the kernel of $q'$, which by \eqref{eqn-relative euler} is given by $\Omega_{\pi}(1)$.
\begin{equation}\label{eqn-morphism of cotangent bundles on projective bundles}
i^*\mc V\, \longrightarrow\, \Omega_{\pi}(1)\, \longrightarrow\, 0\,.
\end{equation}

\begin{proposition}\label{alternate-description-tangent-bundle-sequence}
There is a natural isomorphism $i^*\mc V(-1)\, \cong \,\Omega_{\mb P(V)/X}$. The
composition of this isomorphism with the tensor
product of \eqref{eqn-morphism of cotangent bundles on projective bundles} with $\mc O(-1)$
$$
\Omega_{\mb P(V)/X}\,=\, i^*\mc V (-1) \, \longrightarrow\, \Omega_{\pi}
$$
coincides with the natural surjection of cotangent bundles 
$\Omega_{\mb P(V)/X}\, \longrightarrow\, \Omega_{\pi}\longrightarrow 0$.
\end{proposition}

\begin{proof}
The codimension of a subvariety $A\, \subset\, B$ will be denoted by ${\rm codim}(A,\,B)$.

Let $$A\,\subset\, B\,\subset\, T$$ be smooth varieties such that 
$A$ and $B$ are closed in $T$. Assume that there is a quadruple $(G,\, W,\, s,\, s')$, where
\begin{itemize}
\item $G$ is a vector bundle on $T$ with ${\rm rank}(G)\,=\, 
{\rm codim}(A,\,T)$,

\item $s\,\in\, H^0(T,\, G)$ with the property that its vanishing defines $A$,

\item $W$ is a vector bundle on $B$ with ${\rm rank}(W)\,=\,
{\rm codim}(A,\,B)$, and

\item $s'\,\in\, H^0(B,\, W)$ with the property that its vanishing defines $A$.
\end{itemize}
Further assume that there is a homomorphism $W\,\longrightarrow\, G\big\vert_B$
that takes the section $s'$ to $s\big\vert_B$. Consequently, we have a commutative diagram
\begin{equation}\label{general-normal-bundle-sequence}
\xymatrix{
	G\big\vert_A^\vee\ar[r]^<<<<<\sim\ar[d] & I_{A/T}/I_{A/T}^2\ar[d]\\
	W\big\vert_A^\vee\ar[r]^<<<<<\sim & I_{A/B}/I_{A/B}^2\,.
}
\end{equation}

For convenience $C\times X$ will be denoted by $Y$.
There is the natural inclusion map
$$\mb P(V)\times_Y\mb P(V)\,\hookrightarrow\, \mb P(V)\times_X\mb P(V)$$
and also the diagonal embedding
$$
\mb P(V)\,\hookrightarrow\, \mb P(V)\times_Y\mb P(V),\, \ \ z\, \longmapsto\, (z,\, z).
$$
For $1\, \leq\, i\, <\, j\, \leq\, 3$, let $p_{ij}$ denote the projection
of $C\times \mb P(V)\times_X\mb P(V)$ to the product of its $i$-th and $j$-th
factor. Consider the diagram of homomorphisms on sheaves $C\times \mb P(V)\times_X\mb P(V)$
\begin{equation}\label{define-section-s}
\xymatrix{
		p_{12}^*\mc V\ar[r]\ar@{-->}[dr]& p_{12}^*p_{1,2X}^*V\ar[d]\cong p_{13}^*p_{1,2X}^*V\ar[r]&
		p_{12}^*i_*\mc O(1)\ar[r] &0\\
		& p_{13}^*i_*\mc O(1).}
\end{equation}
The dotted arrow gives a global section
\begin{equation}\label{s0}
s\, \in\, H^0(\mb P(V)\times_X\mb P(V),\,
p_{23*}(p_{12}^*\mc V^\vee\otimes p_{13}^*i_*\mc O(1))), .
\end{equation}
It is straightforward to check that this $p_{23*}(p_{12}^*\mc V^\vee\otimes p_{13}^*i_*\mc O(1))$
is locally free. Let $$Z\,\subset\, \mb P(V)\times_X\mb P(V)$$ be the 
vanishing locus of $s$ in \eqref{s0}.

We will prove that $Z$ is the diagonal in $\mb P(V)\times_X\mb P(V)$. 

To prove this we will first show that $Z$ is contained in $\mb P(V)\times_Y\mb P(V)$.
For this consider the commutative diagram
	\[\xymatrix{
		\mb P(V)\times_X\mb P(V)\ar[r]^{j}\ar[d]_{q_2} & 
		C\times \mb P(V)\times_X\mb P(V)\ar[r]^<<<<<{p_{12}}\ar[d]_{p_{13}} & 
		C\times \mb P(V)\ar[d]\\
		\mb P(V)\ar[r]^i & C\times \mb P(V)\ar[r] & C\times X;
	}
	\] 
here $j(a,\,b)\,=\,(\pi_C(b),\,a,\,b)$ and $q_2(a,\,b)\,=\,b$. Pullback
	\eqref{define-section-s} along the map $j$. It is easy to
see in this pullback that $p_{13}^*i_*\mc O(1)$ is a line bundle supported 
on $\mb P(V)\times_X\mb P(V)$ and $p_{12}^*i_*\mc O(1)$
is a line bundle supported on $\mb P(V)\times_Y\mb P(V)$.
Further restricting this pullback to $Z$ we get a surjection
\begin{equation}\label{e4}
p_{12}^*i_*\mc O(1)\big\vert_Z\, \longrightarrow\, p_{13}^*i_*\mc O(1)\big\vert_Z\,,
\end{equation}
which shows that $Z$ is contained $\mb P(V)\times_Y\mb P(V)$
and the homomorphism in \eqref{e4} is an isomorphism. Consequently,
the restrictions to $Z$ of the two projection maps 
\begin{equation}\label{e5}
\overline{q}_1,\, \overline{q}_2\,\,:\,\, 
\mb P(V)\times_Y\mb P(V)\, \longrightarrow\, \mb P(V)
\end{equation}
actually coincide, which means that $Z$ is contained in the diagonal.

Conversely, it is easily checked that the section $s$
in \eqref{s0} vanishes on the diagonal in $\mb 
P(V)\times_X\mb P(V)$. This proves the assertion that $Z$ is the diagonal in $\mb 
P(V)\times_X\mb P(V)$.

Let $q_1,\, q_2\, :\, \mb P(V)\times_X\mb P(V)\, \longrightarrow\, \mb P(V)$
denote the two projections.
The sheaf $p_{12}^*\mc V^\vee\otimes p_{13}^*i_*\mc O(1)$ is 
supported on the image of $j$ and it is locally free, in fact, it 
is isomorphic to $j^*p_{12}^*\mc V^\vee\otimes q_2^*\mc O(1)$. 
Thus, we may identify $p_{23*}(p_{12}^*\mc V^\vee\otimes p_{13}^*i_*\mc O(1))$
with $j^*p_{12}^*\mc V^\vee\otimes q_2^*\mc O(1)$.
Denote $j^*p_{12}^*\mc V^\vee\otimes q_2^*\mc O(1)$ on $\mb P(V)\times_X\mb P(V)$ by $G$.
It is easily checked that the restriction of
$s\,\in\, H^0(\mb P(V)\times_X\mb P(V),\, G)$ (see \eqref{s0}) to $\mb P(V)\times_Y\mb P(V)$ 
factors as 
$$\overline{q}_2^*\mc O(-1)\,\stackrel{s'}{\longrightarrow}\,
\overline{q}_1^*\Omega_\pi^\vee(-1)\,\longrightarrow\, \overline{q}_1^*i^*\mc V^\vee\,$$
(see \eqref{e5}).
	Let $W$ on $\mb P(V)\times_Y\mb P(V)$ be the locally free
	sheaf $\overline{q}_2^*\mc O(1)\otimes \overline{q}_1^*\Omega_\pi^\vee(-1)$.
	The vanishing locus of the section $s'$ is precisely the diagonal.
	Restricting to the diagonal, and using \eqref{general-normal-bundle-sequence} 
	it follows that there is a commutative diagram
	\[\xymatrix{
		i^*\mc V(-1)\ar[r]^\sim\ar[d] & \Omega_{\mb P(V)/X}\ar[d]\\
		\Omega_\pi\ar@{=}[r] & \Omega_\pi
	}
	\]
	This proves the proposition.
\end{proof}

\section{Atiyah sequence}

Let $V$ be a locally free sheaf of rank $r$ over a smooth variety $X$.
Its Atiyah bundle $At(V)\, \longrightarrow\, X$ fits
in the following \textit{Atiyah exact sequence}
\begin{equation}\label{Atiyah-seq}
0\,\longrightarrow\, \ms End(V)\,\longrightarrow\, At(V)
\,\longrightarrow\, T_X\,\longrightarrow\, 0
\end{equation}
(see \cite{Atiyah}).
We recall a construction of \eqref{Atiyah-seq} which will be
used. Let $P_V\,\stackrel{q}{\longrightarrow}\,X$
denote the principal ${\rm GL}_r(\C)$-bundle associated to $V$.
The differential of $q$ produces an exact sequence on $P_V$
\begin{equation}\label{f1}
0\,\longrightarrow\, K\,:=\, T_{P_V/X}\,\longrightarrow\, T_{P_V}
\,\stackrel{dq}{\longrightarrow}\, q^*T_X\,\longrightarrow\, 0
\end{equation}
Applying $q_*$ to it and then taking ${\rm GL}_r(\C)$-invariants we get
\eqref{Atiyah-seq}.

We have ${\mathcal O}_X\, \subset\, \ms End(V)$; the quotient
$ad(V)\,:=\, \ms End(V)/{\mathcal O}_X$ is identified with the sheaf of
endomorphisms of $V$ of trace zero. Define $at(V)\,=\, At(V)/{\mathcal O}_X$.
Taking the pushout of \eqref{atiyah-seq}
along the quotient map $\ms End(V)\,\longrightarrow\, ad(V)$ we get an exact sequence
\begin{equation}\label{atiyah-seq}
0\,\longrightarrow\, ad(V)\,\longrightarrow\, at(V)\,\longrightarrow\, T_X
\,\longrightarrow\, 0\,.
\end{equation}
We will need an alternate description of \eqref{atiyah-seq}. 
Consider the projective bundle $\mb P(V)\,\stackrel{\pi}{\longrightarrow}\, X$
for $V$, and let
\begin{equation}\label{tangent-bundle}
0\,\longrightarrow\, T_{\mb P(V)/X}\,\longrightarrow\, T_{\mb P(V)}
\,\stackrel{d\pi}{\longrightarrow}\, \pi^*T_X\,\longrightarrow\, 0
\end{equation}
be the exact sequence on $\mb P(V)$ given by the differential $d\pi$.

\begin{lemma}\label{lemma-atiyah seq}
The sequence \eqref{atiyah-seq} coincides with the one obtained by
applying $\pi_*$ to \eqref{tangent-bundle}.
\end{lemma}

\begin{proof} 
Consider the commutative diagram 
\begin{equation}\label{cd1-atiyah-seq}
\xymatrix{
P_V\times \mb P^{r-1}\ar[r]^{\sim}\ar[rd]_{p_1} & \mb P(q^*V)\ar[r]^{\t q}\ar[d]_{\t \pi} & \mb P(V)\ar[d]^{\pi}\\
			& P_V\ar[r]^{q} & X
		}
	\end{equation}
($p_1$ is the projection to the first factor). Here the first isomorphism  follows from the fact that $q^*V$ is a trivial vector bundle \cite[Example 4.2.3 and 4.2.6]{HL}. The differentials of the maps
in it produce the following commutative diagram (without the dotted arrow)
	\[
	\xymatrix{
		& \t q^*T_{\mb P(V)/X}\ar[r]\ar[d] &\t q^* T_{\mb P(V)}\ar[r]\ar@{=}[d] & \t q^*\pi^*T_X\\
		\t \pi^*K\ar[r]\ar@{=}[d] & T_{\mb P(q^*V)}\ar[r]^\gamma\ar[d]_\delta & \t q^*T_{\mb P(V)}\ar[d]\\
		\t \pi^*K\ar[r] &\t \pi^*T_{P_V}\ar[r]^\theta\ar@{-->}[ru]^\beta & \t \pi^*q^*T_X\ar@{=}[ruu]
	}
	\]
in which the rows and columns are short exact sequences.
As $\mb P(q^*V)\,\cong\, P_V\times \mb P^{r-1}$, it follows that $\delta$ has a section
$s\,:\,\t \pi^*T_{P_V}\,\longrightarrow\, T_{\mb P(q^*V)}$. Define $\beta\,:=\,\gamma\circ s$,
and consider the commutative diagram of short exact sequences 
\begin{equation}\label{cd2-atiyah-sequence}
		\xymatrix{
			0\ar[r] &\t \pi^*K\ar[r]\ar[d]^{\beta_0} & 
				\t \pi^*T_{P_V}\ar[r]^\theta\ar@{-->}[d]^\beta & \t 
				\pi^*q^*T_X\ar[r]\ar[d]^{\sim} &0\\
			0\ar[r] &\t q^*T_{\mb P(V)/X}\ar[r] &\t q^* T_{\mb P(V)}\ar[r] &\t q^*\pi^*T_X\ar[r] &0\,.
		}
	\end{equation}

We will prove that $\beta$ is surjective and ${\rm GL}_r(\C)$-equivariant. 

For this it suffices to verify the assertion over an open subset 
$U\,\subset\, X$ on which $V$ is trivial. In this case \eqref{cd1-atiyah-seq} becomes 
\begin{equation*}
\xymatrix{
U \times{\rm GL}_r(\C)\times \mb P^{r-1}\ar[r]^<<<<{\t q}\ar[d]_{p_1\sim\t \pi} & U\times \mb P^{r-1}\ar[d]^{\pi}\\
U \times{\rm GL}_r(\C) \ar[r]^{q} & U
}
\end{equation*}
where $\t q(u,\,A,\,[w])\,=\, (u,\,[Aw])$.
Let $(t,\,X,\,0)$ be an element in the fiber $\t \pi^*T_{P_V}\big\vert_{(u,A,[w])}$;
it is sent to $(t,\, \overline{Xw})$ by $\t q$, where 
$\overline{Xw}\,\in\, T_{\mb P^{r-1}}\big\vert_{[Aw]}$ is the image of 
$Xw\,\in\, T_{\C^{r}\setminus 0}\big\vert_{Aw}$ under the natural
map $$\C^{r}\setminus \{(0,0\ldots,0)\} \,\longrightarrow\, \mb P^{r-1}.$$ It
is straightforward to see that the map 
$$\beta\big\vert_{(u,A,[w])}\, :\, \t \pi^*T_{P_V}\big\vert_{(u,A,[w])}\,\longrightarrow\,
\t q^* T_{\mb P(V)}\big\vert_{(u,A,[w])}\,\cong\, T_{\mb P(V)}\big\vert_{(u,[Aw])}$$
is surjective.

The ${\rm GL}_r(\C)$ action on $U \times{\rm GL}_r(\C)\times \mb P^{r-1}$ is given
by $$(u,\,A,\,[w])\cdot\textbf{g}\,=\,(u,\,A\textbf{g},\,[\textbf{g}^{-1}w])\,.$$
This action sends any $(t,\,X,\,0)\,\in\, \t \pi^*T_{P_V}\big\vert_{(u,A,[w])}$
to $(t,\,X\textbf{g},\,0)\,\in\, \t \pi^*T_{P_V}\big\vert_{(u,A\textbf{g},[\textbf{g}^{-1}w])}$.
Clearly, both these get mapped to the same vector in $T_{\mb P(V)}\big\vert_{(u,[Aw])}$.
This completes the proof of the assertion that $\beta$ 
is ${\rm GL}_r(\C)$-equivariant and surjective. 

Next we will show that $\t \pi_*(\beta)$ is surjective. Again, this would be done locally. 
For any $(u,\,A)\,\in\, U\times {\rm GL}_r(\C)$,
the bundle $(\t q^*T_{\mb P(V)})\big\vert_{(u,A)\times \mb P^{r-1}}$
is simply $T_{U,u}\otimes \mc O_{\mb P(V_u)}\oplus T_{\mb P(V_u)}$. Therefore, the dimension 
$$h^0((u,A)\times \mb P^{r-1},\,(\t q^*T_{\mb P(V)})\big\vert_{(u,A)\times \mb P^{r-1}})$$
does not change. Thus, by Grauert's theorem the canonical map 
$$(\t\pi_*(\t q^*T_{\mb P(V)}))\big\vert_{(u,A)}\,\stackrel{\sim}{\longrightarrow}
\,h^0((u,A)\times \mb P^{r-1},(\t q^*T_{\mb P(V)})\big\vert_{(u,A)\times \mb P^{r-1}})$$
is an isomorphism. Since $\t \pi_*(\t\pi^*T_{P_V})\,\cong\, T_{P_V}$,
to prove that $\t \pi_*\beta$ is surjective it suffices to show that
\begin{equation}\label{m1}
T_{P_V}\big\vert_{(u,A)}\,\longrightarrow\, H^0((u,A)\times \mb P^{r-1},
\,(\t q^*T_{\mb P(V)})\big\vert_{(u,A)\times \mb P^{r-1}})
\end{equation}
is surjective. Take any $(t,\,X)\,\in\, T_{P_V}\big\vert_{(u,A)}$; the map in \eqref{m1}
sends it to a pair consisting of $t$ and a vector field on $\mb P^{r-1}$.
Computing as above, the vector field assigns to the point $[w]\,\in\, \mb P^{r-1}$ 
the tangent vector $\overline{XA^{-1}w}\,\in\, T_{\mb P^{r-1},[w]}$. 
Vector fields on $\mb P(V_u)$ are naturally identified with 
${\rm End}(V_u)/\langle \lambda\cdot{\rm Id}\rangle\,=\,ad(V_u)$.
Thus, the map in \eqref{m1} sends $(t,\,X)$ to $(t,\,\overline{XA^{-1}})$.
Consequently, \eqref{m1} is surjective and its kernel consists of
pairs of the form $(0,\,\lambda A)$, which implies that the kernel
is one dimensional. This proves that $\t \pi_*(\beta)$ 
is surjective with its kernel being a line bundle on $P_V$. 

The locally free sheaf $K$ in \eqref{f1} is identified with $P_V\times M_r(\C)$;
note that $M_r(\C)$ is the Lie algebra of ${\rm GL}_r(\C)$. 
The map $K\,\longrightarrow\, T_{P_V}$ has the following local description. Consider 
an open set $U\, \subset\, X$ over which $V$ is trivialized.
For any $(u,\,A)\,\in\, U\times{\rm GL}_r(\C)$,
the map
$$T_{{\rm GL}_r(\C),{\rm Id}}\,\cong\, K\big\vert_{(u,A)}
\,\longrightarrow\, T_{P_V}\big\vert_{(u,A)}\,\cong\, T_{U,u}\oplus T_{{\rm GL}_r(\C),A}$$ 
is identified with the map defined by $X\,\longmapsto\, (0,\,AX)$. As $\t \pi_*(\beta)$ is surjective, 
it follows from Snake Lemma for $\t \pi_*$ applied on \eqref{cd2-atiyah-sequence} 
that the kernels of 
$\t \pi_*(\beta_0)$ and $\t \pi_*(\beta)$ are the same. Consequently,
the kernel of $\t \pi_*(\beta)$ is precisely the trivial bundle $P_V\times \C$
sitting inside $P_V \times M_r(\C)$ as scalar matrices. 
	
Apply $q_*\t\pi_*$ to \eqref{cd2-atiyah-sequence} and take the ${\rm GL}_r(\C)$ 
invariants. The top row of the resulting sequence is the one in
\eqref{Atiyah-seq}, while the bottom row is obtained by applying 
$\pi_*$ to \eqref{tangent-bundle}. The map $q_*(K)\,\longrightarrow\, \pi_*(T_{\mb P(V)/X})$
is identified with the canonical map $\ms End(V)\,\longrightarrow\, ad(V)$.
This proves the lemma. 
\end{proof}

Let
\begin{equation}\label{a1}
[At(V)]\, \in\, {\rm Ext}^1(T_X,\,{\ms End}(V))
\end{equation}
denote the class of the extension in \eqref{Atiyah-seq}.

\begin{lemma}\label{class of Atiyah seq}
Let $f\,:\,V\,\longrightarrow\, V'$ be a morphism between vector bundles
$V,\, V'$ on $X$. Then the image of the cohomology class $[At(V)]$ in \eqref{a1} under
the natural map
$$
{\rm Ext}^1(T_X,\,{\ms End}(V))\,
\,\xrightarrow{\,\, f\circ \_\,\,}\,\, {\rm Ext}^1(T_X,\,{\ms Hom}(V,\,V'))
$$
coincides with the image of $[At(V')]$ under the natural map
$$
{\rm Ext}^1(T_X,\,{\ms End}(V'))\,\,\xrightarrow{\,\,\_ \circ f\,\,}\,\, 
{\rm Ext}^1(T_X,\,{\ms Hom}(V,\,V'))\,.
$$
\end{lemma}

\begin{proof}
We will recall a description of the image of $[At(V)]$ under the isomorphism 
\begin{equation}\label{eqn-atiyah class}
{\rm Ext}^1(T_X,\,{\ms End}(V))\,\cong \,{\rm Ext}^1( V,\,V\otimes \Omega_X),
	\end{equation}
where $\Omega_X\,=\, T^*_X$. The ideal sheaf of
the  diagonal $$\Delta\,\subset\, X\times X$$ will be denoted by $\mc I$.
Let $p_1,\,p_2\,:\,X\times X\,\longrightarrow\, X$ be the two natural projections. 
Tensoring the exact sequence
$$0\,\longrightarrow\, \mc I/\mc I^2\,\longrightarrow\, \mc O_{X\times X}/\mc I^2
\,\longrightarrow\, \mc O_{\Delta}\,\longrightarrow\, 0\,.$$
with $p_1^*V$, and then applying $p_{2*}$, we get an exact sequence
$$0 \,\longrightarrow\, \Omega_X\otimes V \,\longrightarrow\, p_{2*}(p_1^*V\otimes
\mc O_{X\times X}/\mc I^2) \,\longrightarrow\, V \,\longrightarrow\, 0\,.$$
The extension class of this sequence is $-[At(V)]$ under the isomorphism
in \eqref{eqn-atiyah class};
see \cite[Theorem 5]{Atiyah}. Now we have a natural diagram
\begin{equation}
\begin{tikzcd}\label{eqn-diagram of atiyah classes}
0 \ar[r] & V\otimes \Omega_X \ar[r] \ar[d,"f \otimes {\rm Id}"] &
p_{2*}(p_1^*V\otimes \mc O_{X\times X}/\mc I^2) \ar[r] \ar[d] &
V \ar[r] \ar[d,"f"] & 0\\
0 \ar[r] & V'\otimes \Omega_X \ar[r] &
p_{2*}(p_1^*V'\otimes \mc O_{X\times X}/\mc I^2) \ar[r] &
V' \ar[r] & 0
\end{tikzcd}
\end{equation}
where the top row (respectively, bottom row) corresponds to $-[At(V)]\, \in\,
{\rm Ext}^1(V,\,V\otimes \Omega_X)$ (respectively, 
$-[At(V')]\, \in\, {\rm Ext}^1(V',\,V'\otimes \Omega_X)$). Consequently, the pushout
of the top row of (\ref{eqn-diagram of atiyah classes}) by the morphism
${\rm Id}\otimes f$ is same as the pullback of the bottom row by the morphism $f$.
In other words, the image of the class of top row under the map
$${\rm Ext}^1(V,\,V\otimes \Omega_X)\,\,\xrightarrow{\,\,(f\otimes{\rm Id})\,\, \circ \_\,}\,
\,{\rm Ext}^1(V,\,V'\otimes \Omega_X)$$
coincides with the image of the class of the bottom row under the map
$${\rm Ext}^1(V',\,V'\otimes \Omega_X)\,\,\xrightarrow{\,\,\_ \circ f\,\,}\,\,
{\rm Ext}^1(V,\,V'\otimes \Omega_X)\,.$$
Using this together with the canonical identification
$${\rm Ext}^1(V,\,V'\otimes \Omega_X)\,\cong\, {\rm Ext}^1(T_X,\,{\ms Hom}(V,\,V'))$$ 
the lemma follows.
\end{proof}

\section{The relative Atiyah sequence}\label{section relative Atiyah sequence}

Let $X$ be a smooth projective variety, $C$ a smooth projective curve
and $V$ a vector bundle on $C\times X$. Let 
\begin{equation}\label{p2}
p_C\,:\,C\times X\,\longrightarrow\, C\ \ \text{ and }\ \ p_X\,:\,C\times X
\,\longrightarrow\, X
\end{equation} 
be the natural projections. Pulling back, along the inclusion map
$p_C^*T_C\,\hookrightarrow\, p_C^*T_C \oplus p_X^*T_X$, 
of the Atiyah exact sequence for $V$
\begin{equation}\label{f4}
0\,\longrightarrow\, {\ms End}(V)\,\longrightarrow\, At(V)
\,\longrightarrow\, p_C^*T_C \oplus p_X^*T_X\,\longrightarrow\, 0
\end{equation}
we get the relative Atiyah sequence
\begin{equation}\label{eqn-relative atiyah exact seq}
0 \,\longrightarrow\, {\ms End}(V)\,\longrightarrow\, At_C(V)
\,\longrightarrow\, p_C^*T_C \,\longrightarrow\, 0 \,.
\end{equation}
The pushout of \eqref{eqn-relative atiyah exact seq} along the projection
$\ms End(V)\,\longrightarrow\, ad(V)$ produces an exact sequence
\begin{equation}\label{eqn rel ad-Atiyah}
0\,\longrightarrow\, ad(V)\,\longrightarrow\, at_C(V)
\,\longrightarrow\, p_C^*T_C\,\longrightarrow\, 0
\end{equation}
on $C\times X$.
Henceforth, \eqref{eqn rel ad-Atiyah} will be referred to as the {\it relative
adjoint Atiyah} sequence. 

Let $\pi\,:\,\mb P(V)\,\longrightarrow\, C\times X$ denote the projective 
bundle for $V$; define
$$
\pi_C\, :=\,p_C \circ\pi\, :\,\mb P(V)\,\longrightarrow\, C\, ,
$$
where $p_C$ is the projection in \eqref{p2}.

The tangent bundle sequence for the maps
$\mb P(V)\,\stackrel{\pi}{\longrightarrow}\, C\times X\,\stackrel{p_X}{\longrightarrow}\, X$
(see \eqref{p2}) produced an exact sequence
\begin{equation}\label{eqn-tangent bundle sequence}
0\,\longrightarrow\, T_\pi\,\longrightarrow\, T_{\mb P(V)/X}
\,\longrightarrow\, \pi_C^*T_C\,\longrightarrow\, 0.
\end{equation}

The following lemma is similar to Lemma \ref{lemma-atiyah seq}.

\begin{lemma}\label{alternate description rel ad-atiyah}
The sequence in \eqref{eqn rel ad-Atiyah} coincides with the
one obtained by applying $\pi_*$ to the sequence
in \eqref{eqn-tangent bundle sequence}.
\end{lemma}

\begin{proof}
Applying $\pi_*$ to \eqref{eqn-tangent bundle sequence}
we get the top row of the natural commutative diagram
\begin{equation}\label{eqn-pushforward of tangent bundle sequence}
\xymatrix{
0 \ar[r]& {ad}(V)\ar[r]\ar@{=}[d]& \pi_*T_{\mb P(V)/X}\ar[r]\ar[d]& 
p_C^*T_C\ar[r]\ar[d] & 0\\
0 \ar[r]& {ad}(V)\ar[r]& \pi_*T_{\mb P(V)}\ar[r]& p_C^*T_C\oplus p_X^*T_X\ar[r] & 0\,.
}
\end{equation}
By Lemma \ref{lemma-atiyah seq} 
the lower sequence in \eqref{eqn-pushforward of tangent bundle sequence}
is the pushout of \eqref{f4}
by the quotient map ${\ms End}(V)\,\longrightarrow\, ad(V)$. Therefore, from
the commutativity of \eqref{eqn-pushforward of tangent bundle sequence}
it follows that the top row of it
is the pushout of \eqref{eqn-relative atiyah exact seq}
by the quotient map ${\ms End}(V)\,\longrightarrow\, ad(V)$. 
\end{proof}

Let $f\,:\,Y\,\longrightarrow\, X$ be a morphism of smooth 
projective varieties. Define
$$F\,:=\,{\rm Id}_C\times f\,:\,C\times Y\,\longrightarrow\, C\times X.$$ 
The following is a consequence of Lemma \ref{alternate description rel ad-atiyah}.

\begin{corollary}\label{base change relative ad-atiyah}
The relative adjoint Atiyah sequence for $F^*V$ coincides with the
one obtained by applying $F^*$ to \eqref{eqn rel ad-Atiyah}.
\end{corollary}

\begin{proof}Let 
$$
\textbf{p}_C\, :\, C\times Y\,\longrightarrow\, C
$$
be the natural projection. Note that we have an exact sequence as in \eqref{eqn-tangent bundle sequence} for $Y$. Consider the Cartesian square
\begin{equation}
\xymatrix{
	\mb P(F^*V)\ar[r]^{F'}\ar[d]_{\pi'} & \mb P(V)\ar[d]^{\pi} \\
	C\times Y\ar[r]^F & C\times X
}
\end{equation}
It is easily checked that applying $F^*$ to the sequence 
$$0\,\longrightarrow\, \pi_*(T_\pi)\,\longrightarrow\, \pi_*(T_{\mb P(V)/X})
\,\longrightarrow\, p_C^*T_C\,\longrightarrow\, 0$$
the sequence 
$$0\,\longrightarrow\, \pi'_*(T_{\pi'})\,\longrightarrow\, \pi_*(T_{\mb P(F^*V)/Y})
\,\longrightarrow\, \textbf{p}_C^*T_C\,\longrightarrow\, 0$$
is obtained.
By Lemma \ref{alternate description rel ad-atiyah} the first sequence is the relative adjoint Atiyah sequence for $F^*V$ 
and the second one is the relative adjoint Atiyah sequence for $V$.
This completes the proof of the corollary.
\end{proof}

Let
$$
[At_C(V)]\, \in\, {\rm Ext}^1(p_C^*T_C,\,{\ms End}(V))
$$
be the class of \eqref{eqn-relative atiyah exact seq}. The next result
follows immediately from Lemma \ref{class of Atiyah seq}.

\begin{corollary}\label{cor-compatibility of atiyah classes}
Let $f\,:\,V\,\longrightarrow\, V'$ be a morphism between vector bundles
$V,\, V'$ on $X$. Then the image of $[At_C(V)]$ under the natural map
$${\rm Ext}^1(p_C^*T_C,\,{\ms End}(V))\,\,
\xrightarrow{\,\,f\circ \_\,\,}\,\, {\rm Ext}^1(p_C^*T_C,\,{\ms Hom}(V,\,V'))$$
coincides with the image of $A(V')$ under the natural map
$${\rm Ext}^1(p_C^*T_C,\,{\ms End}(V'))\,\xrightarrow{\,\,\_ \circ f\,\,} 
\,{\rm Ext}^1(p_C^*T_C,\,{\ms Hom}(V,\,V'))\,.$$
\end{corollary}

The following lemma will be used later.
Consider the Cartesian square
\begin{equation}\label{f5}
	\xymatrix{
		\mb P(g^*V)\ar[r]^{g'}\ar[d]_{\pi'} & \mb P(V)\ar[d]^{\pi} \\
		Z\ar[r]^g & C\times X
	}
\end{equation}

\begin{lemma}\label{base change tangent bundle sequence-1}
	Applying $g^*\pi_*$ (see \eqref{f5}) to the exact sequence
	$$0\,\longrightarrow\, T_\pi\,\longrightarrow\, T_{\mb P(V)/X}
	\,\longrightarrow\, \pi^*p_C^*T_C\,\longrightarrow\, 0$$ yields
	the exact sequence
	$$0\,\longrightarrow\, \pi'_*(g'^*T_{\pi})\,\longrightarrow\,
	\pi'_*(g'^*T_{\mb P(V)/X})\,\longrightarrow\, g^*p_C^*T_C\,\longrightarrow\, 0\,,$$
	where $p_C$ is the projection in \eqref{p2}.
\end{lemma}

\begin{proof}
Since \eqref{eqn-tangent bundle sequence} is an exact sequence of vector bundles, after applying $g'^*$ it remains exact
$$0\,\longrightarrow\, g'^*T_{\pi} \,\longrightarrow\,
	g'^*T_{\mb P(V)/X}\,\longrightarrow\, g'^*\pi^*p_C^*T_C\cong \pi'^*g^*p_C^*T_C\,\longrightarrow\, 0\,.$$
The fibres of $\pi'$ are projective spaces and restriction of  $g'^*T_{\pi}$ to a fibre is the tangent bundle of the corresponding projective space. Therefore, its first cohomology vanishes, and by cohomology and base change theorems we have  $R^1\pi'_*(g'^*T_\pi)=0$. Applying $\pi'_*$ to this exact sequence, we get the exactness of the bottom row in the statement of the lemma.

Now note that we have the following diagram induced by base change maps:
\[\begin{tikzcd}
0 \ar[r] & g^*\pi_* T_\pi \ar[r] \ar[d] & g^*\pi_* T_{\mb P(V)/X}  \ar[r]  \ar[d] & g^*\pi_*\pi^*p_C^*T_C  \ar[r] \ar[d] & 0 \\
0 \ar[r] & \pi'_*(g'^* T_\pi)       \ar[r] & \pi'_*(g'^* T_{\mb P(V)/X})                         \ar[r] & \pi'_*(g'^* \pi^*p_C^*T_C )    \ar[r]  & 0
\end{tikzcd}\]

For any $s\in C\times X$, the fibre of $\pi$ is  a projective space $\mb P(V|_s)$  and  the restriction of the vector bundles $T_{\pi}$ and $\pi^*p_C^*T_C$ to this fibre 
are given by the tangent bundle $T_{\mb P(V|_s)}$ and the trivial bundle $T_{C,p_C(s)}\otimes \mc O_{\mb P(V|_s)}$ respectively. In particular, the dimension of space of global sections of both bundles is independent of $s$, which implies that both the base change maps in the above diagram are isomorphisms. Applying five lemma, we get the required isomorphism of exact sequences.
\end{proof}

\section{Infinitesimal deformation map}\label{section infinitesimal deformation}
We continue with the notation of Section \ref{section relative Atiyah sequence}.

Applying the functor $p_{C*}$ to \eqref{eqn-relative atiyah exact seq} produces
a homomorphism
\begin{equation}\label{eqn-infinitesimal deformation map}
\rho\,:\,T_C \,\longrightarrow\, R^1p_{C*}{\ms End}(V)
\end{equation}
on $C$ which is called the {\it infinitesimal deformation} map.
Take any $c\,\in\, C$, and define $V_c\,:=\,V\big\vert_{c\times X}$. Let
\begin{equation}\label{eqn-infinitesimal deformation map at c}
\rho_{c}\,\,:\,\, T_{C,c}\,\longrightarrow\,H^1(X,\, \ms End(V_c))
\end{equation}
be the composition of the homomorphism
$\rho\big\vert_{T_{C,c}}\,:\,T_{C,c}\,\longrightarrow\, R^1p_{C*}{\ms End}(V)\big\vert_c$,
obtained by restricting $\rho$ in \eqref{eqn-infinitesimal deformation map}
to $c\,\in\, C$, with the natural map
$$R^1p_{C*}{\ms End}(V)\big\vert_c \,\longrightarrow\,H^1(X,\, \ms End(V_c)).$$
The homomorphism $\rho_c$ in \eqref{eqn-infinitesimal deformation map at c}
is called the {\it infinitesimal deformation} map, at $c$, for the family of vector
bundles $V$ on $X$ parametrized by $C$.

Let $\pi_c\,:\,\mb P(V_c)\,\longrightarrow\, X$ be the projective bundle
for the above vector bundle $V_c$. Denote
$$
F\,:=\, {\rm Id}_C\times \pi_c \, :\,C\times \mb P(V_c)
\,\longrightarrow\, C\times X\, ,
$$ 
and let $i_c\,:\,c\times \mb P(V_c)\,\longrightarrow\, C\times \mb P(V_c)$ be the
natural inclusion map. Let $\mc K$ be the kernel of the quotient map
$$F^*V\,\longrightarrow\, i_{c*}i^*_cF^*V\,\cong\, i_{c*}\pi^*_cV_c
\,\longrightarrow\, i_{c*}\mc O_{\mb P(V_c)}(1)\,\longrightarrow\, 0\,,$$
so it fits in the exact sequence
\begin{equation}\label{eqn-sequence of C times P(E_c)}
0 \,\longrightarrow\, \mc K\,\longrightarrow\, F^*V\,\longrightarrow\,
i_{c*}\mc O_{\mb P(V_c)}(1) \,\longrightarrow\, 0\,.
\end{equation}
From this one easily computes that 
${\rm det}(\mc K)={\rm det}(F^*V)\otimes p_C^*\mc O_C(-c)$.

Applying $i_{c}^*$ to (\ref{eqn-sequence of C times P(E_c)}) 
we get a right exact sequence
$$i_c^*\mc K\,\longrightarrow\, i_c^*F^*V\,\cong\, \pi_c^*V_c\,\longrightarrow\,
\mc O_{\mb P(V_c)}(1)\,\longrightarrow\, 0\,.$$
This produces a surjection
\begin{equation}\label{eqn-map of cotangent bundles}
i_c^*\mc K\,\longrightarrow\, \Omega_{\pi_c}(1)\,.
\end{equation}
The kernel of this map is a line bundle. 
Taking determinants and using 
${\rm det}(\mc K)={\rm det}(F^*V)\otimes p_C^*\mc O_C(-c)$
it is easily checked that the kernel 
in \eqref{eqn-map of cotangent bundles} is 
$$\mc O_{\mb P(V_c)}(1)\otimes p_C^*\mc O_C(-c)\big\vert_{c}\,=\,
	\mc O_{\mb P(V_c)}(1)\otimes_\C T_{C,c}^\vee\,.$$
Therefore, we get an exact sequence 
\begin{equation}\label{eqn-map of cotangent bundles-1}
0 \,\longrightarrow\, \mc O_{\mb P(V_c)}(1)\otimes_\C T_{C,c}^\vee\,\longrightarrow\, i_c^*\mc K
\,\longrightarrow\, \Omega_{\pi_c}(1)\,\longrightarrow\, 0\,.
\end{equation}
Dualizing \eqref{eqn-map of cotangent bundles-1} and twisting with 
$\mc{O}_{\mb P(V_c)}(1)$ produces an exact sequence
\begin{equation}\label{eqn-tangent sequence restricted to P(E_c)}
0\,\longrightarrow\, T_{\pi_c}\,\longrightarrow\, i_c^*\mc K^{\vee}(1)
\,\longrightarrow\, \mc O_{\mb P(V_c)}\otimes_\C T_{C,c}\,\longrightarrow\, 0\,.
\end{equation}
Note that the restriction of $T_{\pi_c}$ to any fibre is the tangent bundle of a projective space,
and hence its first cohomology vanishes. Therefore by cohomology and base change theorems we get $R^1\pi_{c*}T_{\pi_c}=0$.
Applying $\pi_{c*}$ to \eqref{eqn-tangent sequence restricted to P(E_c)}
we get the sequence
\begin{equation}\label{eqn-pushforward of tangent sequence restricted to P(E_c)}
0 \,\longrightarrow\, ad(V_c)\,\longrightarrow\, \pi_{c*}i_c^*\mc K^{\vee}(1)
\,\longrightarrow\, \mc O_X\otimes_\C T_{C,c}\,\longrightarrow\, 0
\end{equation}
on $X$. 
The long exact sequence of cohomologies for
\eqref{eqn-pushforward of tangent sequence restricted to P(E_c)} gives a map
\begin{equation}\label{eqn-required map}
H^0(X,\,\mc O_X)\otimes_\C T_{C,c}\,\longrightarrow\, H^1(X,\,ad(V_c))\,.
\end{equation}

\begin{lemma}\label{lemma-infinitesimal deformation map}
The image of the map in \eqref{eqn-required map} coincides 
with the image of the following composition
of homomorphisms:
\begin{equation}\label{infinitesimal deformation using ad}
T_{C,c}\,\xrightarrow{\,\,\rho_{c}\,\,}\, H^1(X, \,\ms End(V_c)) 
\,\longrightarrow\, H^1(X, \,\ms End(V_c)/{\mathcal O}_X)\,=\,
H^1(X, \, ad(V_c))\, ,
\end{equation}
where $\rho_c$ is the homomorphism in \eqref{eqn-infinitesimal deformation map at c}.
\end{lemma}
\begin{proof}
	The images of the maps in \eqref{eqn-required map} and 
	\eqref{infinitesimal deformation using ad} correspond 
	to extension classes determined by two short exact sequences 
	on $X$. Thus, to prove the Lemma, it suffices
	to show that the corresponding short exact sequences can be 
	identified with each other. It is straightforward to see, 
	using the definition of $\rho_c$ in \eqref{eqn-infinitesimal deformation map at c}, 
	that the image of the map in \eqref{infinitesimal deformation using ad} 
	corresponds to the restriction of the short exact sequence 
	\eqref{eqn rel ad-Atiyah} to $c\times X$.
	Thus, it suffices to show that the restriction of the short exact sequence 
	\eqref{eqn rel ad-Atiyah} to $c\times X$ is identified with 
	the short exact sequence 
	\eqref{eqn-pushforward of tangent sequence restricted to P(E_c)}.
	 
Using Lemma \ref{alternate description rel ad-atiyah} 
the sequence \eqref{eqn rel ad-Atiyah} is identified with
$$0\,\longrightarrow\, \pi_*T_\pi\,\longrightarrow\, \pi_*T_{\mb P(V)/X}
\,\longrightarrow\, \pi_*\pi_C^*T_C\,\longrightarrow\, 0\,.$$
Applying Lemma \ref{base change tangent bundle sequence-1}
by taking $g$ (in equation \eqref{f5}) to be the inclusion map 
$c\times X\,\hookrightarrow\, C\times X$, 
it follows that the restriction of \eqref{eqn rel ad-Atiyah} to $c\times X$ is
\begin{equation}\label{seq lemma infinitesimal def}
0 \,\longrightarrow\, \pi_{c*}(T_{\pi}\big\vert_{\mb P(V_c)})\,\longrightarrow\, 
\pi_{c*}(T_{\mb P(V)/X}\big\vert_{\mb P(V_c)})
\,\longrightarrow\, \pi_{c*}(p_C^*T_{C}\big\vert_{\mb P(V_c)})\longrightarrow 0\,.
\end{equation} 
It now remains to identify this sequence with 
\eqref{eqn-pushforward of tangent sequence restricted to P(E_c)}.

In Proposition \ref{alternate-description-tangent-bundle-sequence} it was
proved that on $\mb P(V)$, the natural surjection of cotangent bundles 
$\Omega_{\mb P(V)/X}\, \longrightarrow\, \Omega_{\pi}$, coincides 
with the map of bundles 
$i^*\mc V (-1) \, \longrightarrow\, \Omega_{\pi}$ 
in \eqref{eqn-morphism of cotangent bundles on projective bundles}.
It is easily checked that, the restriction of the map 
$i^*\mc V\,\longrightarrow\, \Omega_{\pi}(1)$ in 
\eqref{eqn-morphism of cotangent bundles on projective bundles}
to $\mb P(V_c)$, is the map $i_c^*\mc K\,\longrightarrow\, \Omega_{\pi_c}(1)$, 
that was obtained in \eqref{eqn-map of cotangent bundles}. Putting these two together, we see 
that the restriction, 
$\Omega_{\mb P(V)/X}\big\vert_{\mb P(V_c)}\, \longrightarrow\, 
\Omega_{\pi}\big\vert_{\mb P(V_c)}$ coincides with 
the map $i_c^*\mc K(-1)\,\longrightarrow\, \Omega_{\pi_c}$, 
obtained after twisting \eqref{eqn-map of cotangent bundles}.
Taking dual we see that the map 
$T_{\pi}\big\vert_{\mb P(V_c)}\,\longrightarrow\, T_{\mb P(V)/X}\big\vert_{\mb P(V_c)}$
is identified with $T_{\pi_c}\,\longrightarrow\, i_c^*\mc K^{\vee}(1)$,
which appears in \eqref{eqn-tangent sequence restricted to P(E_c)}.
Now applying $\pi_{c*}$ we see that \eqref{seq lemma infinitesimal def}
is identified with \eqref{eqn-pushforward of tangent sequence restricted to P(E_c)}.
This completes the proof of the Lemma.
\end{proof}

\begin{remark}\label{remark non-split}
	We summarize what we have done above as follows. 
	Given a triple $(c,X,V)$, where $V$ is a vector bundle on 
	$C\times X$, and $c\in C$ is a closed point, we produced a short exact sequence 
	\eqref{eqn-pushforward of tangent sequence restricted to P(E_c)}
	on $X$.
	We shall refer to this short exact sequence as ${\rm SES}(c,X,V)$.
	In Lemma \ref{lemma-infinitesimal deformation map} we proved
	that ${\rm SES}(c,X,V)$ is the same as restricting the relative 
	adjoint Atiyah sequence of $V$ to $c\times X$.
	We also produced another 
	triple $(c,X_1,V_1)$, where $X_1=\mb P(V_c)$ and $V_1=\mc K$. 
	The next Proposition combined with Lemma \ref{lemma-infinitesimal deformation map}
	shows that ${\rm SES}(c,X_1,V_1)$ is non-split on $X_1$.
	We will use this in the proof of Lemma \ref{lemma-cohomology computation end(G_d)}.
\end{remark}

We have the following Proposition due to Narasimhan and Ramanan.

\begin{proposition}\label{Narasimhan-Ramanan}
	The infinitesimal deformation map of $\mc K$ 
	(see \eqref{eqn-sequence of C times P(E_c)}) at $c$, that is,
	the map in \eqref{eqn-infinitesimal deformation map at c}
	$$T_{C,c}\stackrel{\rho_c}{\longrightarrow} H^1(\mb P(V_c),\ms End(i_c^*\mc K))$$
	is injective.
\end{proposition}

\begin{proof}
	If the vector bundle $\mc F$ is defined by the cocycle $\{a_{\alpha\beta}\}$
	on a variety $X$, then the Atiyah bundle $At(\mc F)$ is given by the cocycle 
	$\{-a_{\alpha\beta} da_{\alpha\beta}\}$; this is well known, for example,
	see \cite[\S~4.4]{Biswas-Raghavendra}. It is easily checked that, after 
	identifying $\ms End(\mc F)$ with $\ms End(\mc F^\vee)$ using the transpose 
	map, the Atiyah bundle $At(\mc F^\vee)$ is given by the cocycle 
	$\{a_{\alpha\beta} da_{\alpha\beta}\}$. Thus, the classes $At(\mc F)$
	and $At(\mc F^\vee)$ differ by a minus sign in 
	$H^1(X,\ms End(\mc F)\otimes \Omega_X)$.
	
	If we take $X=C\times \mb P(V_c)$ and apply the above to $\mc K$,
	then it is clear that the images of $At(\mc K)$
	and $At(\mc K^\vee)$ differ by a minus sign in 
	$H^1(\mb P(V_c),\ms End(i_c^*\mc K))$.
	From this it follows
	that the infinitesimal deformation map at $c$ for $\mc K$ is injective
	if and only if the infinitesimal deformation map at $c$ for $\mc K^\vee$ is injective.
	The injectivity for $\mc K^\vee$ is proved in \cite[Proposition 4.4]{NR}.
\end{proof}

\section{Canonical bundle of $S_d$}\label{section canonical bundle S_D}

In this section we begin by recalling a construction from \cite[\S~4]{GS}. 
The main result of this section is Lemma \ref{lemma-canonical bundle of S_D}.
All assertions in this section, before 
Lemma \ref{lemma-canonical bundle of S_D}, can be proved by 
using minor modifications of 
the proofs in \cite[\S~4,\,\S~5]{GS}.

As before, $C$ is a complex projective curve and $E\, \longrightarrow\, C$
a vector bundle of rank $r$. Let $D\,\in \,C^{(d)}$. 
We fix an ordering of the points of $D$ 
\begin{equation}\label{define-sequence-c_i}
(c_1,\,c_2,\,\cdots,\,c_d)\, \in\, C^d\, .
\end{equation}
We will use this ordering to inductively construct a variety $S_j$ and a vector bundle
$A_j\, \longrightarrow\, C\times S_j$ for all $1\, \leq\, j\, \leq\, d$.

Set $S_{0}\,=\,{\rm Spec}\,\C$ and $A_{0}\,=\,E$. For $j\geqslant 1$, we will define
$(S_{j},\, A_{j})$ assuming that the pair $(S_{j-1},\, A_{j-1})$ has been defined. 
Let 
$$\alpha_{j-1}\,:\,\{c_j\}\times S_{j-1}\,\hookrightarrow\, C\times S_{j-1} $$ 
be the natural closed immersion, where $c_j$ is the point in \eqref{define-sequence-c_i}.
Consider the projective bundle
\begin{equation}\label{def f_j,j-1}
f_{j,j-1}\,:\, S_{j}\,:=\,\mathbb{P}(\alpha^{*}_{j-1}A_{j-1})
\,\longrightarrow\, S_{j-1}
\end{equation}
and define the map
\begin{equation}
F_{j,j-1}\,:=\,{\rm Id}_C\times f_{j,j-1}\,:\,C\times S_{j}
\,\longrightarrow\, C\times S_{j-1}\,.
\end{equation}
We also have the natural closed immersion 
$$i_{j}\,:\,\{c_j\}\times S_{j}\,\hookrightarrow\, C\times S_{j}\,.$$
Let $p_{1,j}\, :\, C\times S_{j} \,\longrightarrow\, C$
and $p_{2,j}\, :\, C\times S_{j} \,\longrightarrow\, S_{j}$
be the natural projections. For each $j$, we have the following diagram
\[
\begin{tikzcd}
	\{c_j\}\times S_{j} \arrow[r,hook,"i_j"] \arrow[rd, "="]& C\times S_{j} \arrow[r,"F_{j,j-1}"] 
	\arrow[d,"p_{2,j}"] & C\times S_{j-1} \arrow[d,"p_{2,j-1}"] \\
	& S_j \arrow[r,"f_{j,j-1}"] & S_{j-1} \arrow[u, bend right=90, "\alpha_{j-1}", labels=below right]
\end{tikzcd}
\]
Let $\mathcal{O}_{j}(1) \,\longrightarrow\, S_{j}$ denote the universal line bundle.
Then over $C\times S_{j}$ we have the homomorphisms
\begin{align}\label{quotient-1}
F^{*}_{j,j-1}A_{j-1}\,\longrightarrow\, &\,\, (i_{j})_{*}i^{*}_{j}F^{*}_{j,j-1}A_{j-1} \\
	=\, &\,\, (i_{j})_{*}f^{*}_{j,j-1}\alpha_{j-1}^*A_{j-1} \nonumber\\
	\longrightarrow\, &\,\, (i_{j})_{*}\mathcal{O}_{j}(1)\nonumber
\end{align}
Define $A_{j}$ to be the kernel of the composition of
homomorphisms $F^{*}_{j,j-1}A_{j-1}\,\longrightarrow\,(i_{j})_{*}\mathcal{O}_{j}(1)$
in \eqref{quotient-1}. 
Thus, we have the following short exact sequence on $C\times S_j$
\begin{equation}\label{def-A_j}
	0\,\longrightarrow\, A_j\,\longrightarrow\, F_{j,j-1}^*A_{j-1}
	\,\longrightarrow\, i_{j*}(\mc O_j(1))\,\longrightarrow\, 0\,.
\end{equation}
For any $x\notin \{c_j\}\times S_{j} \subset C\times S_j$ the localisations of $A_j$ and $F_{j,j-1}^*A_{j-1}$ at $x$ are same. In particular $A_j$ is locally free at $x$.
Now assume $x\in \{c_j\}\times S_{j} \subset C\times S_j$
and let $R$ be the local ring of $C\times S_j$ at $x$. Then the localisation of $i_{j*}(\mc O_j(1))$ is given by $R/rR$ for some regular element $r\in R$ and hence its depth is given by 
${\rm dim}~R-1$. By Auslander–Buchsbaum formula we get that projective dimension of $i_{j*}(\mc O_j(1))$ is $1$. This implies that $A_j$ is locally free at $x$. 
Therefore $A_j$ is locally free on $C\times S_j$. 
Thus, $(S_j,\, A_j)$ is constructed.

For $d\, \geqslant\, j\,>\, i\,\geqslant\, 0$, define morphisms
\begin{align}
f_{j,i}\,=\,f_{j,j-1}\circ \ldots \circ f_{i+1,i}\,& :\,S_{j}
\,\longrightarrow\, S_{i}, \label{f10}\\
F_{j,i}\,={\rm Id}_C\times f_{j,i}=\,F_{j,j-1}\circ \ldots \circ F_{i+1,i}\,\,& :\,
C\times S_{j}\,\longrightarrow\, C\times S_{i}\,.\nonumber
\end{align}
Note that both the morphisms are flat.

Closed points of $S_d$ are in bijective correspondence with the filtrations
\begin{equation*}
E_d\,\subset\, E_{d-1} \,\subset\, E_{d-2} \,\subset\, \cdots \,\subset\,
E_1\,\subset\, E_0\,=\,E,
\end{equation*}
where each $E_j$ is a locally free sheaf of rank $r$ on $C$
and $E_j/E_{j+1}$ is a skyscraper sheaf of degree one supported at 
$c_{j+1}\,\in \, C$.

\begin{remark}\label{remark-permutation}
	The following is not difficult to check and as we will not be using this fact, we omit the proof. 
	Let $(b_1,b_2,\ldots,b_d)$ be a permutation of $(c_1,c_2,\ldots,c_d)$
	and let $(S_d',A_d')$ denote the space built inductively 
	as above using the $b_i$. Then there is an isomorphism $\theta:S_d\stackrel{\sim}{\to}S_d'$
	such that under the isomorphism $Id_C\times \theta:C\times S_d\to C\times S_d'$
	the bundles $A_d$ and $A_d'$ are isomorphic. 
	This shows that upto an isomorphism, the pair $(S_d,A_d)$ depends only 
	on the effective divisor $D=\sum_{i=1}^d[c_i]\in C^{(d)}$ 
	defined by $(c_1,c_2,\ldots,c_d)$. 
\end{remark}


Let $p_1\,:\,C\times S_d\,\longrightarrow\, C$ denote the natural projection.
We will construct a quotient of $p_1^*E$.
Using the flatness of $F_{d,i}$, and pulling back \eqref{def-A_j},
for $i=0,\ldots,d-1$, we get a sequence of inclusions 
\begin{equation}\label{filtration on C times S_d}
A_{d}\,\subset\, F_{d,d-1}^{*}A_{d-1}\,\subset\, F_{d,d-2}^{*}A_{d-2}\,\subset\, \cdots 
\,\subset\, F_{d,1}^{*}A_{1}\,\subset\, F_{d,0}^*A_0\,=\,p^{*}_{1}E. 
\end{equation}
Define 
$$
B_{j}^d \,:=\, p^{*}_{1}E/F_{d,j}^{*}A_{j}\, \cong\, F_{d,j}^{*}(p^{*}_{1}E/A_j)\, ,
$$
so $B_{d}^d$ fits in the exact sequence
\begin{equation}\label{univ-quotient-S_D}
0\,\longrightarrow\, A_d\,\longrightarrow\, p^{*}_{1}E\,\longrightarrow\,
B^{d}_{d}\,\longrightarrow\, 0\,.
\end{equation}
The sheaf $B^{d}_{d}$ is $S_d$--flat.

Pulling back the exact sequence \eqref{def-A_j} 
on $C\times S_j$ along $F_{d,j}$ 
and using flatness of $F_{d,j}$, 
we see that
$$F_{d,j-1}^{*}A_{j-1}/F_{d,j}^{*}A_j
\,\,\cong \,\,F_{d,j}^*(i_j)_*(\mc O_j(1))\,\,\cong\,\, f_{d,j}^*(\mc O_j(1))\,.$$ 
Thus, for each $j$ there is an exact sequence on $C\times S_d$
\begin{equation*}
0 \,\longrightarrow\, f_{d,j}^*(\mc O_j(1))\,\longrightarrow\, B_{j}^d
\,\longrightarrow\, B_{j-1}^d\,\longrightarrow\, 0\,.
\end{equation*}
The support of $B^d_d$ is finite over $S_d$. Consider the
natural projection $p_2\,:\,C\times S_d\,\longrightarrow\, S_d$, and 
apply $p_{2*}$ to the above sequence; taking the top exterior products,
we have 
\begin{equation}\label{det-B^d_d}
\det (p_{2*}(B^d_d))\,=\,\bigotimes_{j=1}^{d}f_{d,j}^*\mc O_j(1)\,.
\end{equation}

Consider the Hilbert-Chow map
\begin{equation}\label{hc}
\phi\,\,:\,\,\mc Q\,\longrightarrow\, C^{(d)}.
\end{equation} 
The universal property of $\mathcal{Q}$ and \eqref{univ-quotient-S_D} yields a morphism 
$g_d\,:\,S_d\,\longrightarrow\, \mathcal{Q}$. 
Let $\mc Q_D$ denote the scheme theoretic fiber of $\phi$ over the point
$D\,\in\, C^{(d)}$. It is straightforward to
check that $g_d$ factors as
\begin{equation}\label{f7}
g_d\,:\,S_d\,\longrightarrow\, \mc Q_D
\end{equation}
(the same notation $g_d$ is used for the map it is factoring through).

\begin{lemma}\label{lemma-canonical bundle of S_D}
There is a line bundle $\mc L$ on $\mc Q_D$ (see \eqref{f7})
such that the canonical bundle $K_{S_d}$ of $S_d$ is isomorphic to $g_d^*\mc L$.
\end{lemma}

\begin{proof}
We will compute $K_{S_d}$ by induction on $d$. Recall the
exact sequence \eqref{def-A_j} on $C\times S_j$ defining $A_j$
\begin{equation*}
0\,\longrightarrow\, A_j\,\longrightarrow\, F_{j,j-1}^*A_{j-1}
\,\longrightarrow\, i_{j*}(\mc O_j(1))\,\longrightarrow\, 0\,.
\end{equation*}
The rightmost sheaf is a line bundle on the divisor $c_j\times S_j$ and so 
\begin{equation}\label{det-A_j}
\det (A_j)\,=\, \det (F_{j,j-1}^*A_{j-1})\otimes p_C^*\mc O_C(-c_j)\,,
\end{equation}
where $p_C\,:\,C\times S_j\,\longrightarrow\, C$ denotes the projection.
Since $S_j$ is obtained as a tower of projective bundles, it 
follows that the restriction of a line bundle $\mc L$ on 
$C\times S_j$ to $\{c\}\times S_j$ is independent of the point $c$. 
Restricting \eqref{det-A_j} to $c_j\times S_j$ we get that
\begin{align*}
\det \left(A_j\big\vert_{c_j\times S_j}\right)&\,=\,
(F_{j,j-1}^*{\rm det}(A_{j-1}))\big\vert_{c_j\times S_j}\\
&=\,f_{j,j-1}^*(\det (A_{j-1})\big\vert_{c_{j}\times S_{j-1}})\\
&=\,f_{j,j-1}^*\left(\det (A_{j-1}\big\vert_{c_{j-1}\times S_{j-1}})\right).
	\end{align*}
The last equality makes sense only when $j\, >\,1$. However, when $j\,=\,1$, 
$$\det (A_1\big\vert_{c_1\times S_1})\,=\,f_{1,0}^*(\det (A_{0})\big\vert_{c_{1}\times S_{0}})\,,$$
and the bundle on the right-hand side is trivial. Thus, by descending induction
on $j$ we conclude that $\det (A_j\big\vert_{c_j\times S_j})$ is trivial.
This also shows that $\det (A_j\big\vert_{c\times S_j})$ is trivial for all $c$.

For a locally free sheaf of rank $r$ on a smooth variety $X$
and the corresponding projective bundle $\pi\,:\,\mb P(E)\,\longrightarrow\, X$ ,
$$K_{\mb P(E)}\,=\,\pi^*(\det (E)\otimes K_X)\otimes \mc O_{\mb P(E)}(-r)\,.$$
Setting $X\,=\,S_{j-1}$ and $E\,=\,A_{j-1}\big\vert_{c_j\times S_{j-1}}$,
we conclude that
\begin{equation*}
K_{S_j}\,=\,f_{j,j-1}^*K_{S_{j-1}}\otimes \mc O_j(-r)
\end{equation*}
(note that $\det(E)$ is trivial). Using induction 
we get the first equality in the following.
\begin{equation*}
K_{S_d}\,=\,\bigotimes_{j=1}^{d}f_{d,j}^*\mc O_j(-r)=(\det (p_{2*}(B^d_d))^{-1})^{\otimes r}\,. 
\end{equation*}
The second equality follows using equation \eqref{det-B^d_d}.
It follows from \cite[Lemma 3.1]{GS-nef} that $p_{2*}(B^d_d)$
is the pullback
along $g_d$ (see \eqref{f7}) of a locally free sheaf from $\mc Q$. Thus, the line bundle 
${\rm det}(p_{2*}(B^d_d))$ is the pullback along $g_d$ of a line bundle from $\mc Q$,
and hence it is the pullback of a line bundle from $\mc Q_D$.
This completes the proof of the Lemma.
\end{proof}

\begin{corollary}\label{cor-cohomology over Q_D and S_D}
Let $V$ be a locally free sheaf on $\mc Q_D$. Then $H^i(\mc Q_D,\,V)\,\cong\,
H^i(S_d,\,g_d^*V)$.
\end{corollary}

\begin{proof}
As $g_d$ is birational (see \cite[Proposition 5.13]{GS}), and $\mc Q_D$
is normal (see \cite[Corollary 6.6]{GS}), it follows that 
$g_{d*}(\mc O_{S_d})\,=\,\mc O_{\mc Q_D}$. From \cite[Theorem 4.3.9]{Laz}
it follows that $R^qg_{d*}(K_{S_d})\,=\,0$ for all $q\,>\,0$. Using
Lemma \ref{lemma-canonical bundle of S_D}
this shows that $R^qg_{d*}(\mc O_{S_d})\otimes \mc L\,=\,0$ for $q\,>\,0$, that is,
$R^qg_{d*}(\mc O_{S_d})\,=\,0$ for $q\,>\,0$. Now the result follows using the 
Leray spectral sequence.
\end{proof}

\section{Projective bundle computations}

We collect here some facts which shall be used later.
Let $E$ be a locally free sheaf of rank $r$ on a scheme 
$T$ and $\mb P(E)\,\stackrel{\pi}{\longrightarrow}\,T$ the corresponding
projective bundle. Then we have an exact sequence on $\mb P(E)$
\begin{equation}\label{f8}
0\,\longrightarrow\, \Omega_{\pi}(1)\,\longrightarrow\, \pi^*E
\,\longrightarrow\, \mc O_{\mb P(E)}(1)\,\longrightarrow\, 0\,.
\end{equation}

\begin{lemma}\label{basic-facts-projective-bundles}
With notation as above,
\begin{enumerate}
\item $R^i\pi_*(\Omega_{\pi}(1))\,=\,0$\,\ $\forall\ i$,

\item $H^i(\mb P(E),\, \Omega_{\pi}(1))\,=\,0$\,\ $\forall\ i$,

\item $R^i\pi_*(\mc O_{\mb P(E)}(-1))\,=\,0$\,\ $\forall\ i$,

\item $H^i(\mb P(E), \,\mc O_{\mb P(E)}(-1))\,=\,0$\,\ $\forall\ i$,

\item $H^i(\mb P(E),\, \pi^*E(-1))\,=\,0$\,\ $\forall\ i$,

\item $H^i(\mb P(E),\,\mc O_{\mb P(E)})\,\stackrel{\sim}{\longrightarrow}\,
H^{i+1}(\mb P(E),\, \Omega_{\pi})$\,\ $\forall\ i$,

\item $H^i(\mb P(E),\, \pi^*E^\vee(1)\otimes \Omega_\pi)\,=\,0$\,\ $\forall\ i$,

\item $H^i(\mb P(E),\,\ms End(\Omega_\pi))\,\stackrel{\sim}{\longrightarrow}\,
H^{i+1}(\mb P(E),\, \Omega_{\pi})$\,\ $\forall\ i$,
\end{enumerate}
\end{lemma}

\begin{proof}
Note that for any $t\,\in\, T$ from the Euler sequence on the projective space $\mb P(E_t)$
$$0 \,\longrightarrow \,\Omega_{\mb P(E_t)}(1)\,\longrightarrow\, E_t\otimes \mc O_{\mb P(E_t)}
\,\longrightarrow \,\mc O_{\mb P(E_t)}(1)\,\longrightarrow\, 0$$
it follows that 
$$H^i(\mb P(E_t),\,\Omega_{\mb P(E_t)}(1))\,=\,0\,\ \ \ \forall\ \ i\,\geq\, 0\,.$$
Also we have that $H^i(\mb P(E_t),\,\mc O_{\mb P(E_t)}(-1))\,=\,0$ for all
$i\,\geq\, 0$. By Grauert's theorem we get that 
$$R^i\pi_*(\Omega_{\pi}(1))\,=\,R^i\pi_*(\mc O_{\mb P(E)}(-1))\,=\,0
\,\ \ \ \forall\ \ i\,\geq\, 0 \,.$$
Therefore we get (1) and (3). Using the Leray spectral sequence we see
that (2) and (4) follows immediately from (1) and (3) respectively. From projection
formula and (3) we get that 
$$R^i\pi_*(\pi^*E(-1))\,=\,E\otimes R^i\pi_*(\mc O_{\mb P(E)}(-1))\,=\,0
\,\ \ \ \forall\ \ i\,\geq\, 0\,.$$
Therefore (5) again follows from the Leray spectral sequence. Twisting (\ref{f8}) by $\mc O_{\mb P(E)}(1)$ we get an exact sequence
$$0 \,\longrightarrow\, \Omega_\pi \,\longrightarrow\, \pi^*E(-1)
\,\longrightarrow\, \mc O_{\mb P(E)}\,\longrightarrow\, 0\,.$$
Considering the associated long exact sequence of cohomologies, and applying (5), we get (6).
Using projection formula and (1) we get that
$$R^i\pi_*\left(\pi^*E^\vee(1)\otimes \Omega_\pi\right)\,=\,E^{\vee}\otimes R^i\pi_*
\left(\Omega_{\pi}(1)\right)\,=\, 0\,\ \ \ \forall\ \ i\,\geq\, 0\,.$$
Therefore (7) follows from the Leray spectral sequence.
For the last assertion consider the exact sequence
$$0\,\longrightarrow\, \Omega_\pi\,\longrightarrow\, \pi^*E^\vee(1)\otimes \Omega_\pi
\,\longrightarrow\, \ms End(\Omega_\pi)\,\longrightarrow\, 0$$
obtained by tensoring the dual of \eqref{f8} 
with $\Omega_\pi$. Taking the corresponding long exact sequence of cohomologies
and using (7) we see that the statement follows.
\end{proof}

Next let $G$ be a locally free sheaf on $C\times T$ with
$C$ being a smooth projective curve. Set $G_c\,:=\,G\big\vert_{c\times T}$,
and consider the corresponding projective bundle $\mb P(G_c)\,\stackrel{\pi}{\longrightarrow}
\,T$. Define $$F\, :=\,{\rm Id}_C\times \pi \,:\,C\times \mb P(G_c)
\,\longrightarrow\, C\times T.$$
Let $i\,:\,\{c\}\times \mb P(G_c)\,\hookrightarrow\, C\times \mb P(G_c)$ be 
the inclusion map.
On $C\times \mb P(G_c)$ we have the short exact sequence 
\begin{equation}\label{gen-proj-e1}
0\,\longrightarrow\, \mc K\,\longrightarrow\, F^*G\,\longrightarrow\,
i_*\left(\mc O_{\mb P(G_c)}(1)\right)\,\longrightarrow\, 0\,;
\end{equation}
the map on the right is the composition as in \eqref{quotient-1}. It is evident that 
$$
F^*G(-c\times \mb P(G_c))\, :=\,
F^*G\otimes {\mathcal O}_{C\times \mb P(G_c)}(-c\times \mb P(G_c))\,\subset\, \mc K\,.
$$
Restricting \eqref{gen-proj-e1} to $\{c\}\times \mb P(G_c)$ we see that 
the quotient of the above inclusion is $i_*(\Omega_{\pi}(1))$. 
It, furthermore, fits in a commutative diagram 
\begin{equation}\label{e9}
	\xymatrix{
		0\ar[r] & F^*G(-c\times \mb P(G_c))\ar[r]\ar[d] & \mc K\ar[r]\ar[d] &
			 i_*(\Omega_{\pi}(1))\ar[r]\ar@{=}[d] & 0\\
		0\ar[r] & \mc L\ar[r] & i_*\left(\mc K\big\vert_{c\times \mb P(G_c)}\right)
\ar[r] & i_*(\Omega_{\pi}(1))\ar[r] & 0
	}
\end{equation}
The kernel of the left and middle vertical arrows are 
both $\mc K(-c\times \mb P(G_c))$. From this description 
it is clear that the left vertical arrow in \eqref{e9} is the cokernel 
of the inclusion map
$$\mc K(-c\times \mb P(G_c))\,\subset\, F^*G(-c\times \mb P(G_c))$$
which is obtained by tensoring the inclusion map $\mc K\subset F^*G$ 
with $\mc O_{C\times \mb P(G_c)}(-c\times \mb P(G_c))$.
In particular,
\begin{itemize}
\item $\mc L\,\cong\, i_*(\mc O_{\mb P(G_c)}(1))$, and

\item the left vertical arrow in \eqref{e9} is identified with the map 
$F^*G\,\longrightarrow\, i_*(\mc O_{\mb P(G_c)}(1))$ in \eqref{gen-proj-e1}.
\end{itemize}

\begin{lemma}\label{gen-lemma-relation between cohomology of A_j and A_j-1}
Using the above notation, for all $i\,\geqslant\, 0$, and $t\,\in\, C$, the natural maps
\begin{enumerate}
\item $H^i\left(\mb P(G_c),\, F^*G(-c\times \mb P(G_c))\big\vert_{t\times \mb P(G_c)}\right)
\,\longrightarrow\, H^i\left(\mb P(G_c),\, \mc K\big\vert_{t\times \mb P(G_c)}\right)$, and

\item $H^i\left(\mb P(G_c),\, F^*G^\vee\big\vert_{t\times \mb P(G_c)}\right)
\,\longrightarrow\, H^i\left(\mb P(G_c),\, \mc K^\vee\big\vert_{t\times \mb P(G_c)}\right)$
\end{enumerate}
are isomorphisms. 
\end{lemma}

\begin{proof}
The inclusions maps $F^*G(-c\times \mb P(G_c))\,\subset\, \mc K\,\subset\, F^*G$
are isomorphisms outside $c\times \mb P(G_c)$. Thus, it suffices to prove the two
assertions under the assumption that $c\,=\,t$. From the above discussion 
it follows that the map in (1) is identified with the composition of homomorphisms
$$H^i\left(\mb P(G_c),\, F^*G\big\vert_{c\times \mb P(G_c)}\right) 
\,\longrightarrow\, H^i(\mb P(G_c), \,\mc O_{\mb P(G_c)}(1)) \,\longrightarrow\, 
H^i\left(\mb P(G_c), \,\mc K\big\vert_{c\times \mb P(G_c)}\right)\,.$$
From Lemma \ref{basic-facts-projective-bundles}	it follows that both these
maps are isomorphisms.
	
The map $F^*G^\vee\big\vert_{c\times \mb P(G_c)}\,\longrightarrow\,
\mc K^\vee\big\vert_{c\times \mb P(G_c)}$
factors as 
$$F^*G^\vee\big\vert_{c\times \mb P(G_c)}\,\longrightarrow\, T_{\pi}(-1)
\,\longrightarrow\, \mc K^\vee\big\vert_{c\times \mb P(G_c)}\,.$$
Taking cohomology, the second assertion in the lemma follows
using Lemma \ref{basic-facts-projective-bundles}.
\end{proof}

\section{Cohomology of some sheaves on $S_d$}\label{section cohomology of some sheaves}

Recall from Section \ref{section canonical bundle S_D} that associated to a divisor $D\in 
C^{(d)}$ and an ordering $(c_1,c_2,\ldots,c_d)$ of points of $D$ we have the schemes $S_j$ 
for $1\leq j\leq d$. In this section we again fix this divisor $D$. We now choose the 
ordering $(c_1,c_2,\ldots,c_d)$ of $D$ in the following manner: Define $c_1$ to be any 
point in $D$. Now suppose we have defined $c_j$ for $1 \,\leq\, j\,\leq\, d-1$. Then we 
define $c_{j+1}$ to be $c_j$ if $c_j\,\in \,D\setminus \{c_1,c_2,\ldots,c_j\}$. Otherwise 
define $c_j$ to be any point in $D\setminus \{c_1,c_2,\ldots,c_j\}$. Throughout this 
section, we fix this ordering of $D$. To clarify, the above conditions on the ordering do 
not determine the ordering uniquely.

For every $1 \,\leq\, j\,\leq \,d$, let $$\Omega_j\, \longrightarrow\, S_j$$ be the 
relative cotangent bundle of the map $f_{j,j-1}$ in \eqref{f10}. We have the relative 
Euler sequence
\begin{equation}
0\,\longrightarrow\, \Omega_j(1)\,\longrightarrow\, 
f_{j,j-1}^*\alpha^*_{j-1}A_{j-1}\, \xrightarrow{\ s\ }\, \mc O_j(1)\,\longrightarrow\, 0 
\end{equation}
on $S_j$. For the sequence in \eqref{def-A_j} $$0 \,\longrightarrow\, 
A_j\,\longrightarrow\, F_{j,j-1}^*A_{j-1} \,\longrightarrow\, (i_j)_*\mc 
O_j(1)\,\longrightarrow\, 0\, ,$$ since the quotient is supported on $c_j\times S_j$, 
there are the natural inclusion maps $$F_{j,j-1}^*A_{j-1}(-c_j\times S_j) \,\subset\, 
A_{j} \,\ \ \text{ and }\,\ \ F_{j,j-1}^*A_{j-1}^{\vee}\,\subset\, A_j^{\vee}\,.$$ For a 
locally free sheaf $V$ on $C\times S_j$, and an effective divisor $D'$ on $C$, the vector 
bundle $V\otimes p_C^*\mc O_C(-D')$ will be denoted by $V(-D')$ (this involves a mild 
abuse of notation).

\begin{lemma}\label{lemma-relation between cohomology of A_j and A_j-1}
Let $D'$ be an effective divisor on $C$. 
For all $i\,\geqslant\, 0$, and $t\,\in\, C$, the two natural maps
\begin{enumerate}
\item $H^i\left(S_j,\, F_{j,j-1}^*A_{j-1}((-c_j-D')\times S_j)\big\vert_{t\times S_j}\right)
\,\longrightarrow\, H^i\left(S_j,\, A_{j}(-D')\big\vert_{t\times S_j}\right)$

\item $H^i\left(S_j,\,F_{j,j-1}^*A_{j-1}^{\vee}\big\vert_{t\times S_j}\right)
\,\longrightarrow\, H^i\left(S_j,\,A_j^{\vee}\big\vert_{t\times S_j}\right)$
\end{enumerate}
are isomorphisms.
\end{lemma}

\begin{proof}
This follows from Lemma \ref{gen-lemma-relation between cohomology of A_j and A_j-1}
by setting $T\,=\,S_{j-1}$, $G\,=\,A_{j-1}(-D')$ and $c\,=\,c_j$
in it.
\end{proof}

Recall from \eqref{univ-quotient-S_D} that we have an exact sequence 
$$0\,\longrightarrow\, A_d \,\longrightarrow\, F^*_{d,0}A_0\,\longrightarrow\,
B^d_d\,\longrightarrow\, 0$$
on $C\times S_d$, and a filtration
$$A_{d}\,\subset\, F_{d,d-1}^{*}A_{d-1}\,\subset\, F_{d,d-2}^{*}A_{d-2}\,\subset\,
\cdots\,\subset\, F_{d,1}^{*}A_{1}\,\subset\, F_{d,0}^*A_0\,=\,p^{*}_{1}E$$
(see \eqref{filtration on C times S_d}).
Since $B^d_d$ is supported on $D\times S_d$, there is an inclusion map
$$F^*_{d,0}A_0(-D)\,\subset\, A_d\,.$$

\begin{corollary}\label{cor-cohomology of A_d}
For any $t\,\in \, C$, the following statements hold:
\begin{enumerate}
\item the natural map $H^i\left(S_d,\,F^*_{d,0}A_0(-D\times S_d)\big\vert_{t\times S_d}\right)
\,\longrightarrow\, H^i\left(S_d,\,A_d\big\vert_{t\times S_d}\right)$ is an isomorphism,

\item the natural map $H^i\left(S_d,\,F^*_{d,0}A_0^{\vee}\big\vert_{t\times S_d}\right)
\,\longrightarrow\, H^i\left(S_d,\,A_d^{\vee}\big\vert_{t\times S_d}\right)$ is an isomorphism,

\item $H^i\left(S_d,\,A_d\big\vert_{t\times S_d}\right)\,=\,0\,\ \ \forall\,\ i\,>\,0$, and

\item $H^i\left(S_d,\,A_d^{\vee}\big\vert_{t\times S_d}\right)\,=\,0\,\ \ \forall\,\ i\,>\,0$.
	\end{enumerate} 
\end{corollary}

\begin{proof}
First statement (2) will be proved. For $i,\,j\,\geqslant\,0$,
we have a commutative square
\begin{equation*}
\xymatrix{
H^i\left(S_d,\, F_{d,j}^*A_j^\vee\big\vert_{t\times S_d}\right)\ar[r] & 
H^i\left(S_d,\,F_{d,j+1}^*A_{j+1}^\vee\big\vert_{t\times S_d}\right)\\
H^i\left(S_{j+1},\,F_{j+1,j}^*A_j^\vee\big\vert_{t\times S_{j+1}}\right)\ar[r]\ar[u]_{\sim} & 
H^i\left(S_{j+1},\,A_{j+1}^\vee\big\vert_{t\times S_{j+1}}\right)\ar[u]_{\sim}
}
\end{equation*}
The lower horizontal arrow is an isomorphism using Lemma
\ref{lemma-relation between cohomology of A_j and A_j-1}. 
Therefore the upper horizontal arrow is also an isomorphism. 
The composition
$$H^i\left(S_d,\, F_{d,0}^*A_0^\vee\big\vert_{t\times S_d}\right)\,\longrightarrow\, 
H^i\left(S_d,\, F_{d,1}^*A_1^\vee\big\vert_{t\times S_d}\right)\,\longrightarrow\,
\cdots \,\longrightarrow\,
H^i\left(S_d,\, A_d^\vee\big\vert_{t\times S_d}\right)$$
of these upper horizontal isomorphisms for $j\,=\,0,\,1,\,\cdots,\,d-1$ produces
an isomorphism 
$H^i\left(S_d,\, F_{d,0}^*A_0^\vee\big\vert_{t\times S_d}\right)\,\longrightarrow\,
H^i\left(S_d,\, A_d^\vee\big\vert_{t\times S_d}\right)$. This proves (2).

Now note that we have 
$$F_{d,0}^*A_0\vert_{t\times S_d} = p_1^*E\big\vert_{t\times S_d}\,\cong\, E_t\otimes \mc O_{S_d}\,.$$
This and (2) together prove statement (4).

For every $0\,\leqslant\, j\,\leqslant\, d-1$ define the divisor 
$$D_j\,:=\,\sum\limits_{l=j+1}^d c_l$$ 
on $C$. Define $D_{d}=0$. For $d \geqslant i,\,j\,\geqslant\,0$, we have a commutative square
\begin{equation*}
\xymatrix{
H^i\left(S_d,\,F_{d,j}^*(A_j(-D_{j}\times S_{j+1}))\big\vert_{t\times S_d}\right)\ar[r] & 
H^i\left(S_d,\,F_{d,j+1}^*(A_{j+1}(-D_{j+1}\times S_{j+1}))\big\vert_{t\times S_d}\right)\\
H^i\left(S_{j+1},\,F_{j+1,j}^*(A_j(-D_{j}\times S_{j+1}))\big\vert_{t\times S_{j+1}}\right)\ar[r]\ar[u]_{\sim} & 
H^i\left(S_{j+1},\,A_{j+1}(-D_{j+1}\times S_{j+1})\big\vert_{t\times S_{j+1}}\right)\ar[u]_{\sim}
}
\end{equation*}
The lower horizontal arrow is an isomorphism using Lemma
\ref{lemma-relation between cohomology of A_j and A_j-1}. 
Therefore the upper horizontal arrow is also an isomorphism. 
The composition of these upper horizontal arrows for $j\,=\,0,1,\,\cdots,\,d-1$ 
$$H^i\left(S_d,\,F_{d,0}^*(A_0(-D\times S_{1})\big\vert_{t\times S_d}\right)\,\longrightarrow\,
H^i\left(S_d,\,F_{d,1}^*(A_j(-D_{1}\times S_{2})\big\vert_{t\times S_d}\right)
$$
$$
\,\longrightarrow\,\cdots \,\longrightarrow\,
H^i\left(S_d,\,A_d\big\vert_{t\times S_d}\right)$$
produces an isomorphism $H^i\left(S_d,\,F_{d,0}^*(A_0(-D\times S_{1})\big\vert_{t\times S_d}\right)
\,\longrightarrow\,H^i\left(S_d,\,A_d\big\vert_{t\times S_d}\right)$.
This proves (1). 

Again, (3) follows using (1) and the fact that
$F^*_{d,0}A_0(-D\times S_d)\big\vert_{t\times S_d}$ is the trivial bundle.
\end{proof}

For convenience of notation, denote $G_d\,:=\, A_d\big\vert_{c_d\times S_d}$.

\begin{lemma}\label{lemma-cohomology computation end(G_d)}
Using the above notation,
\begin{enumerate}
\item $h^1(S_d,\,\ms End(G_d))\,=\,1$,

\item $h^i(S_d,\,\ms End(G_d))\,=\,0$ for all $i\,\geqslant\,2$,

\item $h^i(S_d,\,G_d^\vee(1))\,=\,0$ for all $i\,\geqslant\,1\,.$
\end{enumerate}
\end{lemma}

\begin{proof}
The lemma will be proved by induction on $d$. 
First set $d\,=\,1$. Then $S_1\,=\,\mb P(E_{c_1})$ and $G_1$ sits in the 
short exact sequence
$$0\,\longrightarrow\, \mc O_{\mb P^{r-1}}(1)\,\longrightarrow\, G_1
\,\longrightarrow\, \Omega_{\mb P^{r-1}}(1)\,\longrightarrow\, 0\,.$$
Since $H^1(\mb P^{r-1},\,T_{\mb P^{r-1}})\,=\,0$, this sequence splits. 
Using this, it is easily checked that all three assertions of the lemma are true 
when $d\,=\,1$.
	
	Let $\delta\geqslant 2$ and assume that the lemma 
	holds for all $1\,\leqslant\, d\,\leqslant\, \delta-1$.
	Set $d\,=\,\delta$. We will first prove (3).
	
	Set $X\,=\,S_{d-1}$ and $V\,=\,A_{d-1}$ on $C\times X$. 
	Then the quotient in \eqref{eqn-sequence of C times P(E_c)}
	is exactly the quotient in \eqref{quotient-1} when $j\,=\,d$.
	Thus, the short exact sequence in \eqref{eqn-sequence of C times P(E_c)}
	is identified with the short exact sequence \eqref{def-A_j}.
	In that case, \eqref{eqn-map of cotangent bundles-1} becomes 
	\begin{equation}\label{eqn-cohomology computation end(A)}
		0\,\longrightarrow\,\mc O_{d}(1)\,\longrightarrow\, G_d
		\,\longrightarrow\, \Omega_d(1)\,\longrightarrow\, 0\,.
	\end{equation}
	
	Let $T_d\, \longrightarrow\, S_d$ denote the relative tangent bundle for the map
	$f_{d,d-1}\,:\,S_d\, \longrightarrow\, S_{d-1}$ (take $j=d$ in 
	\eqref{def f_j,j-1}). Dualizing (\ref{eqn-cohomology computation end(A)}) 
	and after a twist we have an exact sequence on $S_d$
	\begin{equation}\label{f11}
		0\,\longrightarrow\, T_d\,\longrightarrow\, G_d^\vee(1)
		\,\longrightarrow\, \mc O_{S_d}\,\longrightarrow\, 0
	\end{equation}
	This sequence is identified with the short exact sequence in 
	\eqref{eqn-tangent sequence restricted to P(E_c)}.
	Since $$R^if_{d,d-1*}(T_d)\,=\,R^if_{d,d-1*}(\mc O_{S_d})\,=\,0$$ for all
	$i\,\geqslant\, 1$, it follows that $R^if_{d,d-1*}(G_d^\vee(1))\,=\,0$ 
	for $i\,\geqslant\, 1$. Therefore by Leray spectral sequence we get that 
	$$H^i(S_d,G_d^{\vee}(1))=H^i(S_{d-1},f_{d,d-1*}(G_d^{\vee}(1)))\,.$$
	Thus, it suffices to compute dimensions of $H^i(S_{d-1},f_{d,d-1*}(G_d^{\vee}(1)))$.
	Recall that $S_d$ is the projective bundle associated to 
	the bundle $A_{d-1}\big\vert_{c_d\times S_{d-1}}$. Therefore
	the sheaf $f_{d,d-1*}(T_d)$
	is $ad(A_{d-1}\big\vert_{c_d\times S_{d-1}})$. 
	Applying $f_{d,d-1*}$ to (\ref{f11}) we get an exact sequence
	\begin{equation}\label{eqn-pushforward of tangent bundle over S_d}
		0 \longrightarrow ad(A_{d-1}\big\vert_{c_d\times S_{d-1}}) 
\longrightarrow f_{d,d-1*}(G_d^{\vee}(1)) 
\longrightarrow \mc O_{S_{d-1}}\longrightarrow 0\,.
\end{equation}

First consider the case where $c_d\,\neq\, c_{d-1}$. By the choice of the ordering $(c_1,c_2,\ldots,c_d)$ made in the beginning of this section, this means that 
$c_d$ appears in $D$ with multiplicity 1, and hence we can write $D\,=\,c_d+ D'$
such that $c_d$ does not appear in the support of $D'$. The sheaf $A_{d-1}$ on $C\times S_{d-1}$
agrees with $p_C^*E$ outside $D'\times S_{d-1}$. It follows that 
$A_{d-1}\big\vert_{c_d\times S_{d-1}}$ is the trivial bundle. Therefore
$$H^i\left(S_{d-1},\, ad\left(A_{d-1}\big\vert_{c_d\times S_{d-1}}\right)\right)\,=\,0\,\
\ \forall\ \ i\,>\,0\,.$$
{}From the long exact sequence of (\ref{eqn-pushforward of tangent bundle over S_d}) 
it follows that 
$$h^i(S_d,\, G_d^\vee(1))\,=\, h^i(S_{d-1},\, f_{d,d-1*}(G_d^\vee(1)))\,=\, 0 \,
$$
for all $i\,\geqslant\, 1$.
	
Next consider the case where $c_d\,=\,c_{d-1}$. 
Set $Y\,=\,S_{d-2}$, $W\,=\,A_{d-2}$ and $c\,=\,c_{d-1}$.
Then we have a triple $(c,Y,W)$ as in Remark \ref{remark non-split}. 
We observed in Remark \ref{remark non-split} that we get a 
triple $(c,\,Y_1,\,W_1)$ such that ${\rm SES}(c,\,Y_1,\,W_1)$ is non-split. 
We leave it to the reader to check that ${\rm SES}(c,\,Y_1,\,W_1)$
is exactly \eqref{eqn-pushforward of tangent bundle over S_d}.

As a consequence, it follows that the boundary map in the cohomology sequence of 
\eqref{eqn-pushforward of tangent bundle over S_d}, that is,
\begin{equation}\label{boundary map}
H^0\left(S_{d-1},\,\mc O_{S_{d-1}}\right) \,\longrightarrow \,H^1\left(S_{d-1},\,
ad\left(A_{d-1}\big\vert_{c_d\times S_{d-1}}\right)\right)
\end{equation}
is an inclusion. 
By inductive hypothesis, using $c_d\,=\,c_{d-1}$,
and the rationality of $S_{d-1}$, we get that
\begin{align*}
h^1\left(S_{d-1}, ad\left(A_{d-1}\big\vert_{c_d\times S_{d-1}}\right)\right)&=\,
h^1\left(S_{d-1}, \ms End\left(A_{d-1}\big\vert_{c_d\times S_{d-1}}\right)\right)\\
&=\,h^1\left(S_{d-1}, \ms End\left(A_{d-1}\big\vert_{c_{d-1}\times S_{d-1}}\right)\right)\\
&=\,h^1\left(S_{d-1},\,\ms End(G_{d-1})\right)\,=\,1\,,\\
h^i\left(S_{d-1}, ad\left(A_{d-1}\big\vert_{c_d\times S_{d-1}}\right)\right)&=\,
h^i\left(S_{d-1}, \ms End\left(A_{d-1}\big\vert_{c_d\times S_{d-1}}\right)\right)\\
&=\,h^i\left(S_{d-1}, \ms End\left(A_{d-1}\big\vert_{c_{d-1}\times S_{d-1}}\right)\right)\\
&=\,h^i\left(S_{d-1},\,\ms End\left(G_{d-1}\right)\right)\,=\,0
\, \qquad \,\ \forall\ \, i\,\geqslant\, 2\,.
\end{align*}
Consequently, the boundary map in \eqref{boundary map} is an isomorphism.
It follows from the long exact sequence of cohomologies associated to 
\eqref{eqn-pushforward of tangent bundle over S_d} that 
$$h^i(S_{d-1},\, f_{d,d-1*}(G_d^\vee(1)))\,=\,h^i(S_d,G_d^\vee(1))\,=\,0$$
for $i\,\geqslant\, 1$.
This proves the third assertion of the lemma when $d\,=\,\delta$.

Tensoring \eqref{eqn-cohomology computation end(A)} with $G_d^\vee$
we get the sequence
$$0 \,\longrightarrow \,G_d^{\vee}(1)\,\longrightarrow \,\ms End(G_d)
\,\longrightarrow \,G_d^\vee\otimes \Omega_d(1)\,\longrightarrow\, 0\,.$$
	Using (3) we see that 
	$h^i(S_d,\,\ms End(G_d))\,=\,h^i(S_d,\,G_d^\vee\otimes \Omega_d(1))$
	for $i\,\geqslant\, 1$. We have a short exact sequence obtained by 
	tensoring \eqref{f11} with $\Omega_d$
$$0\,\longrightarrow\, \ms End(\Omega_d)\,\longrightarrow\, G_d^\vee\otimes \Omega_d(1)
\,\longrightarrow\, \Omega_d\,\longrightarrow\, 0\,.$$
Since $f_{d,d-1}\,:\,S_d\,\longrightarrow\, S_{d-1}$
in \eqref{def f_j,j-1} is a projective bundle, using Lemma \ref{basic-facts-projective-bundles} it follows that
$$h^i(S_d,\,\ms End(\Omega_d))\,=\,h^i(S_d,\,\mc O_{S_d})\,=\,0\ \ \, \forall\
\, i\,>\,0\,.$$
Therefore, (again using Lemma \ref{basic-facts-projective-bundles}) we get
$$h^i(S_d,\,G_d^\vee\otimes \Omega_d(1))\,=\,h^i(S_d,\,\Omega_d)
\,=\,h^{i-1}(S_d,\,\mc O_{S_d})$$ for $i\,\geqslant\, 1$.
Thus, $h^i(S_d,\,\ms End(G_d))\,=\,h^{i-1}(S_d,\,\mc O_{S_d})$ for $i\,\geqslant\, 1$,
which proves the first two assertions of the lemma for $d\,=\,\delta$.
This completes the proof of the lemma.
\end{proof}

\section{On the geometry of $\mc Q$}

\subsection{Notation}\label{se9.1}

As before, the $d$-th symmetric product of $C$ is denoted by $C^{(d)}$. 
Let 
\begin{align}
p_1\,:\,C\times \mc Q\,\longrightarrow\, C,\ \ p_2\,:\,
C\times \mc Q\,\longrightarrow\, \mc Q,\label{t5}\\
q_1\,:\, C\times C^{(d)}\,\longrightarrow\, C,\ \ 
q_2\,:\,C\times C^{(d)}\,\longrightarrow\, C^{(d)}\label{t6}
\end{align}
denote the natural projections.
Recall that there is a universal exact sequence on $C\times \mc Q$
\begin{equation}\label{universal-seq-C times Q}
0 \,\longrightarrow\, \mc A\,\longrightarrow\, p_1^*E\,\longrightarrow\, \mc B
\,\longrightarrow\, 0\,.
\end{equation}
Let
\begin{equation}\label{f13}
\Sigma\,\subset\, C\times C^{(d)}
\end{equation}
be the universal divisor.
Define
\begin{equation}\label{t7}
\Phi\,:=\,{\rm Id}_C\times \phi\,:\,C\times \mc Q\,\longrightarrow\, C\times C^{(d)},
\end{equation}
where $\phi$ is the Hilbert-Chow map in \eqref{hc}.

\subsection{Direct image of sheaves on $C\times \mc Q$}

\begin{corollary}\label{cor-direct image of structure sheaf}
The following statements hold:
\begin{enumerate}
\item $\Phi_*\mc O_{C\times \mc Q}\,=\, \mc O_{C\times C^{(d)}}$\ and\ 
$R^i\Phi_*\mc O_{C\times \mc Q}\,=\,0\,\ \ \forall\ i\,>\,0$.

\item $\phi_*\mc O_{\mc Q}\,=\,\mc O_{C^{(d)}}$\ and\ $R^i\phi_*\mc O_{\mc Q}
\,=\,0\,\ \ \forall\ i\,>\,0$.
\end{enumerate}
\end{corollary}

\begin{proof}
The fibers of $\phi$ (respectively, $\Phi$) over any point $D\,\in\, C^{(d)}$
(respectively, $(c,\, D)\,\in\, C\times C^{(d)}$) is isomorphic to $\mc Q_D$.
By Corollary \ref{cor-cohomology over Q_D and S_D} we have 
$h^i(\mc Q_D,\,\mc O_{\mc Q_D})\,=\,h^i(S_d,\,\mc O_{S_d})$. 
Since $S_d$ is a tower of projective bundles, it 
follows that $h^0(S_d,\,\mc O_{S_d})\,=\,1$ 
and $h^i(S_d,\,\mc O_{S_d})\,=\,0$ for all 
$i\,>\,0$. As both $\phi$ and $\Phi$ are flat 
morphisms (see \cite[Corollary 6.3]{GS}), the result now follows from Grauert's theorem 
\cite[p.~288--289, Corollary 12.9]{Ha}.
\end{proof}

Let
\begin{equation}\label{f16}
\mc Z \, \subset\, C\times \mc Q
\end{equation}
be the zero scheme of the inclusion map 
${\rm det}(\mc A)\,\hookrightarrow\, \det (p_1^*E)$, where $p_1$ is the map
in \eqref{t5}. From the
definition of $\phi$ it follows immediately that $\Phi^*\Sigma\,=\,\mc Z$. 
In fact, $\mc Z$ sits in the following commutative diagram
in which both squares are Cartesian
\begin{equation}\label{f14}
\xymatrix{
	\mc Z\ar[r]\ar[d] & C\times \mc Q\ar[r]\ar[d]^\Phi & \mc Q\ar[d]^\phi \\
	\Sigma \ar[r] & C\times C^{(d)}\ar[r] & C^{(d)}
}
\end{equation}
(see \eqref{f13} and \eqref{f16}) and the composition of the top horizontal maps is a finite morphism;
the same holds for the composition of the bottom horizontal maps in \eqref{f14}.
The ideal sheaf $\mc O_{C\times \mc Q}(-\mc Z)$ therefore
annihilates $\mc B$ in \eqref{universal-seq-C times Q}, which
in turn produces an inclusion
map $p_1^*E(-\mc Z)\,\subset\, \mc A$, where $p_1$ is the map
in \eqref{t5}. Applying $\Phi_*$ 
and using Corollary \ref{cor-direct image of structure sheaf} an inclusion map
\begin{equation}\label{f15}
q_1^*E(-\Sigma)\,\cong\, \Phi_*[p_1^*E(-\mc Z)] \,\hookrightarrow\, \Phi_*\mc A
\end{equation}
is obtained, where $q_1$ is the map in \eqref{t6}. Also, note that
since the cokernel of $\mc A\to p_1^*E$ is $\mc B$, the kernel of the map $p^*_1E^{\vee} \to \mc A^{\vee}$ 
is $\mc B^{\vee}$. But $\mc B$ is torsion, hence $\mc B^{\vee}=0$.
 Therefore we also have the natural inclusion map
$p_1^*E^{\vee}\,\hookrightarrow\, \mc A^{\vee}$. Applying $\Phi_*$ and
using Corollary \ref{cor-direct image of structure sheaf} we get an inclusion map
$$q_1^*E^{\vee} \,\hookrightarrow\, \Phi_*(\mc A^{\vee})\,.$$

\begin{proposition}\label{propn-direct image of A}
The following statements hold:
\begin{enumerate}
\item The natural map $q_1^*E(-\Sigma)\,\longrightarrow\, \Phi_*\mc A$ is an isomorphism
(see \eqref{f15}).

\item The natural map $q_1^*E^{\vee}\,\hookrightarrow\,\Phi_*(\mc A^{\vee})$ is an isomorphism.

\item $R^i\Phi_*\mc A\,=\,R^i\Phi_*(\mc A^{\vee})\,=\,0$\ for all\, $i\,>\,0$.
\end{enumerate} 
\end{proposition}

\begin{proof}
First consider the map $\Phi_*[p_1^*E(-\mc Z)]\,\longrightarrow\, \Phi_*\mc A$.
Fix $(c,\,D)\,\in \,C\times C^{(d)}$. We will show that the homomorphism
\begin{equation}\label{f12}
H^0\left(c\times \mc Q_D,\, p_1^*E(-\mc Z)\big\vert_{c\times \mc Q_D}\right)
\,\longrightarrow\, H^0\left(c\times \mc Q_D,\, \mc A\big\vert_{c\times \mc Q_D}\right)
\end{equation}
is an isomorphism.

In view of Corollary \ref{cor-cohomology over Q_D and S_D},
showing that \eqref{f12} is an isomorphism is equivalent 
to showing that the map 
\begin{equation*}
H^0\left(c\times S_d,\, p_1^*E(-D)\big\vert_{c\times S_d}\right)\,\longrightarrow\,\,
H^0\left(c\times S_d,\, A_d\big\vert_{c\times S_d}\right)
\end{equation*}
is an isomorphism. But this is precisely the content of 
Corollary \ref{cor-cohomology of A_d} (1) when $i\,=\,0$.
Since $\mc A$ is flat over $C\times C^{(d)}$, using
Grauert's theorem, \cite[Corollary 12.9]{Ha}, it now follows that 
$q_1^*E(-\Sigma)\,\longrightarrow\, \Phi_*\mc A$ is an isomorphism. 

The other two statements follow from 
Corollary \ref{cor-cohomology of A_d} in the same way. 
\end{proof}

\begin{corollary}\label{Phi_*B}
The natural map 
$$q_1^*E\big\vert_\Sigma\,\longrightarrow\, \Phi_*\mc B$$
is an isomorphism, where $\Sigma$ is defined in \eqref{f13}
and $q_1$ (respectively, $\Phi$) is the map in \eqref{t6}
(respectively, \eqref{t7}). Moreover, $$R^i\Phi_*\mc B\,=\,0$$
for all $i\,>\,0$.
\end{corollary}

\begin{proof}
Using projection formula and Corollary \ref{cor-direct image of structure sheaf} it
follows that $$\Phi_*p_1^*E\,=\,q_1^*E\ \ \text{ and }\ \
R^i \Phi_* p_1^*E\,=\,q_1^*E\otimes R^i\Phi_*\mc O_{C\times \mc Q}\,=\,0$$
for all $i\,>\,0$.
Therefore the statement follows immediately by applying $\Phi_*$ to the 
universal exact sequence 
$$0\,\longrightarrow\,\mc A\,\longrightarrow\, p_1^*E
\,\longrightarrow\, \mc B\,\longrightarrow\, 0$$
and using Proposition \ref{propn-direct image of A}.
\end{proof}

\begin{lemma}\label{Hom(B,B)}
The natural map $\mc O_{\mc Z}\,\longrightarrow\, \ms Hom(\mc B,\,\mc B)$,
where $\mc Z$ is defined in \eqref{f16}, is an isomorphism.
\end{lemma}

\begin{proof}
We will first show that $\mc Z$ is an integral and normal scheme. 
First let us show $\mc Z$ is irreducible. 
As the squares in \eqref{f14} are Cartesian, it follows that 
the fibers of the map $\mc Z\, \longrightarrow\, \Sigma$
are the same as the fibers of the $\phi$. These fibers
have the same dimension 
(see \cite[Proposition 6.1]{GS}) and $\Sigma$
is irreducible. It follows $\mc Z$ is irreducible. 

Next consider the locus $U$ where
the map $\phi$ is smooth. The set
$U$ meets each fiber of $\phi$ in an open set whose complement has 
codimension at least two (see \cite[Corollary 5.6, 5.7]{GS}).
It follows that $U_{\mc Z}\,:=\,(C\times U)\cap \mc Z$ 
is an open set in $\mc Z$ where 
$\mc Z\, \longrightarrow\, \Sigma$ is smooth and 
${\rm codim}(\mc Z\setminus U_{\mc Z},\,\mc Z)\,\geqslant\, 2$.

As $\Sigma$ is smooth, it follows that $U_{\mc Z}$ is smooth. 
Thus, $\mc Z$ satisfies 
Serre's conditions $R_0$, $R_1$. That the fibers of $\Phi$ 
(and hence of $\mc Z\, \longrightarrow\, \Sigma$)
are Cohen-Macaulay is proved in \cite[Corollary 6.4]{GS}
(this is not explicitly mentioned in the statement of the 
Corollary, but is mentioned in the proof). 
It follows from \cite[\href{https://stacks.math.columbia.edu/tag/045J}{Tag 045J}]{Stk}
(or see Corollary to \cite[Theorem 23.3]{Mat})
that the $\mc Z$ is Cohen-Macaulay. This shows 
that $\mc Z$ is reduced, integral and normal.

Next we show that $\mc B$ in \eqref{universal-seq-C times Q} is a 
torsionfree 
$\mc O_{\mc Z}$-module. Indeed, if $\mc B'\,\subset\, \mc B$ is a 
torsion $\mc O_{\mc Z}$-module, then
it is also torsion as a $\mc O_{\mc Q}$-module. This is a contradiction as
$\mc B$ is a coherent and flat $\mc O_{\mc Q}$-module. Hence
$\mc B$ is a torsionfree $\mc O_{\mc Z}$-module.

Let ${\rm Spec}(A)\,\subset\, \mc Z$ be an affine open set,
and let $\mc B_A$ denote the module corresponding to $\mc B$. 
We have the inclusion maps 
$$A\,\subset\, {\rm Hom}_A(\mc B_A,\,\mc B_A)\,\subset\, 
{\rm Hom}_{K(A)}(\mc B_A\otimes_AK(A),\,\mc B_A\otimes_AK(A))\,=\,K(A)\,.$$
As $A$ is normal, and ${\rm Hom}_A(\mc B_A,\,\mc B_A)$ 
is a finite $A$-module, it follows that ${\rm Hom}_A(\mc B_A,\,\mc B_A)$ coincides
with $A$. This proves the lemma.
\end{proof}

\begin{theorem}\label{preliminiary ses}
There is a map $\Xi$ that fits in a short exact sequence
\begin{equation*}
0\,\longrightarrow\, ad(q_1^*E\big\vert_{\Sigma})\,
\stackrel{\Xi}{\longrightarrow}\, \Phi_*\ms Hom(\mc A,\,\mc B)
		\,\longrightarrow\, R^1\Phi_*\ms End(\mc A)\, \longrightarrow\, 0
\end{equation*}
on $\Sigma$.

For every $i\,\geqslant\, 1$, there is a natural isomorphism
\begin{equation*}
R^i\Phi_*\ms Hom(\mc A,\,\mc B)\,\,\stackrel{\sim}{\longrightarrow}\,\,
R^{i+1}\Phi_*\ms End(\mc A).
\end{equation*}
\end{theorem}

\begin{proof}
Application of $\ms Hom(\mc A,\, -)$ to the exact 
sequence in \eqref{universal-seq-C times Q}
produces an exact sequence
\begin{equation}\label{preliminary ses e1}
0\,\longrightarrow\, \ms End(\mc A)\,\longrightarrow\, \ms Hom(\mc A,\,p_1^*E)
\,\longrightarrow\, \ms Hom(\mc A,\,\mc B)\,\longrightarrow\, 0\,.
\end{equation} 
Using the projection formula and Proposition
\ref{propn-direct image of A} it follows that
	$$R^i\Phi_*\ms Hom(\mc A,\,p_1^*E)
\,=\,q_1^*E\otimes R^i\Phi_*(\mc A^{\vee})\,=\,0
$$
for all $i\,>\,0$. Therefore, applying $\Phi_*$ to \eqref{preliminary ses e1} produces
an exact sequence
\begin{equation}\label{u1}
0 \,\longrightarrow\, \Phi_*\ms End(\mc A)\,\longrightarrow\, \Phi_*\ms Hom(\mc A,\,p^*_1E)\,
\longrightarrow\,\Phi_*{\ms Hom}(\mc A,\,\mc B)\,\longrightarrow\,
R^1\Phi_*\ms End(\mc A)\,\longrightarrow 0\,\,.
\end{equation} 
Moreover, we get that there is an isomorphism
\begin{equation*}
R^i\Phi_*\ms Hom(\mc A,\,\mc B)\,\stackrel{\sim}{\longrightarrow}\, R^{i+1}\Phi_*\ms End(\mc A)
	\end{equation*}
	for every $i\,\geqslant\, 1$. This proves the second part of the theorem.

To prove the first part of the theorem, it is enough to show that the image of the map 
	$$\Phi_*\ms Hom(\mc A,\,p^*_1E)\,\longrightarrow\, \Phi_*\ms Hom(\mc A,\, \mc B)$$ is isomorphic to 
	$ad(q_1^*E\big\vert_{\Sigma})$.
	
	Consider the commutative diagram
	\begin{equation*}
		\xymatrix{
			0\ar[r]& \ms Hom(p_1^*E, \,\mc A)\ar[r]\ar[d]& \ms Hom(p_1^*E,\, p_1^*E)\ar[r]\ar[d]& 
			\ms Hom(p_1^*E, \,\mc B)\ar[r]\ar[d]& 0\\
			0\ar[r]& \ms End(\mc A)\ar[r]& \ms Hom(\mc A,\,p_1^*E)\ar[r]& 
			\ms Hom(\mc A,\,\mc B)\ar[r]& 0\,.	
		}
\end{equation*}
Using projection formula together with Proposition \ref{propn-direct image of A}
it follows that
$$R^1\Phi_*\ms Hom(p_1^*E, \,\mc A)\,=\,q_1^*E^{\vee}\otimes R^1\Phi_*\mc A\,=\,0\,,$$
$$\Phi_*\ms Hom(\mc A,\,p_1^*E)
\,=\, q_1^*E\otimes \Phi_*(\mc A^{\vee})\,=\,\ms Hom(q_1^*E,\, q_1^*E)\,.$$
Consequently, applying $\Phi_*$ to the preceding diagram we get a diagram 
\begin{equation*}
\xymatrix{
0\ar[r]& \Phi_*\ms Hom(p_1^*E, \,\mc A)\ar[r]\ar[d]& \Phi_*\ms Hom(p_1^*E,\, p_1^*E)\ar[r]\ar[d]^{\cong} & 
\Phi_*\ms Hom(p_1^*E, \,\mc B)\ar[r]\ar[d] & 0\\
0\ar[r]& \Phi_*\ms End(\mc A)\ar[r]& \Phi_*\ms Hom(\mc A,\,p_1^*E)\ar[r]& 
\Phi_*\ms Hom(\mc A,\,\mc B). &	
}
\end{equation*} 
Thus, the image of the homomorphism
$\Phi_*\ms Hom(\mc A,\,p^*_1E)\,\longrightarrow\, \Phi_*{\ms Hom}(\mc A,\,\mc B)$
coincides with the image of the homomorphism
$\Phi_*\ms Hom(p_1^*E,\, \mc B)\,\longrightarrow\, \Phi_*{\ms Hom}(\mc A,\,\mc B)$.
Now consider the following commutative diagram in which the left vertical arrow 
is an isomorphism due to Lemma \ref{Hom(B,B)}
\begin{equation}\label{preliminary ses e2}
\begin{tikzcd}
0 \ar[r] & \mc O_{\mc Z} \ar[r] \ar[d,"\cong"] & {\ms End}\left(p_1^*E\big\vert_{\mc Z}\right) \ar[r] \ar[d] & 
{ad}\left(p_1^*E\big\vert_{\mc Z}\right)
\ar[r] \ar[d] & 0 \\
0 \ar[r] & {\ms Hom}(\mc B,\, \mc B) \ar[r] & {\ms Hom}(p_1^*E,\,\mc B)
\ar[r] & {\ms Hom}(\mc A,\, \mc B) &
\end{tikzcd}
\end{equation}
Applying $\Phi_*$ to it and using Corollary \ref{Phi_*B} we get the diagram
(the right vertical arrow is defined to be $\Xi$)
\begin{equation}\label{preliminary ses e3}
\xymatrix{
0 \ar[r] & \mc O_{\Sigma} \ar[r] \ar[d]^{\cong} & {\ms End}\left(q_1^*E\big\vert_{\Sigma}\right)
\ar[r] \ar[d]^{\cong} & {ad}\left(q_1^*E\big\vert_{\Sigma}\right)
\ar[r] \ar[d]^{\Xi} & 0 \\
0 \ar[r] & \Phi_*{\ms Hom}(\mc B, \,\mc B) \ar[r] & \Phi_*{\ms Hom}(p_1^*E,\,\mc B)
\ar[r] & \Phi_*{\ms Hom}(\mc A, \,\mc B) &
}
\end{equation}
Therefore, the image of the homomorphism
$\Phi_*{\ms Hom}(p_1^*E,\,\mc B)\,\longrightarrow\, \Phi_*{\ms Hom}(\mc A, \,\mc B)$ is same as the image of the map
${ad}(q_1^*E\big\vert_{\Sigma})\stackrel{\Xi}{\longrightarrow} \Phi_*{\ms Hom}(\mc A,
\,\mc B)$. But the above diagram shows that this map is an inclusion. Hence
its image is isomorphic to 
${ad}\left(q_1^*E\big\vert_{\Sigma}\right)$.
This completes the proof of the theorem.
\end{proof}

Note that it follows from Theorem \ref{preliminiary ses}
that the sheaf $R^i\Phi_*\ms End(\mc A)$ 
is supported on $\Sigma$ for every $i\,\geqslant\, 1$.
By Corollary \ref{cor-direct image of structure sheaf},
the canonical map $R^1\Phi_*\ms End(\mc A)\,\longrightarrow\, R^1\Phi_*ad(\mc A)$
is an isomorphism.
Recall the relative adjoint Atiyah sequence (see \eqref{eqn rel ad-Atiyah}) for the 
locally free sheaf $\mc A$ on $C\times \mc Q$
\begin{equation}\label{rel Atiyah A on C times Q}
0\,\longrightarrow\, ad(\mc A)\,\longrightarrow\, at_C(\mc A)
\,\longrightarrow\, p_C^*T_C\,\longrightarrow\, 0\,.
\end{equation}
Applying $\Phi_*$ to \eqref{rel Atiyah A on C times Q} we get a map of sheaves
\begin{equation}\label{f18}
q_1^*T_C\,\longrightarrow\, R^1\Phi_*ad(\mc A)
\end{equation}
on $C\times C^{(d)}$.

For ease of notation, let
\begin{equation}\label{bar-q_1}
\bq_1\,\, :\,\, \Sigma\,\longrightarrow\,C
\end{equation}
be the composite 
$\Sigma\,\hookrightarrow\, C\times C^{(d)}\,\stackrel{q_1}{\longrightarrow}\,C$,
where $q_1$ is the map in \eqref{t6}.

\begin{theorem}\label{description R1-Phi-End(A)}
The map in \eqref{f18} induces an isomorphism
$$ q_1^*T_C\big\vert_{\Sigma}\,=\,
\bq_1^*T_C\,\stackrel{\sim}{\longrightarrow}\, R^1\Phi_*ad(\mc A)\, ,$$
where $\Phi$ is the map in \eqref{t7}.
Moreover, $R^i\Phi_*\ms End(\mc A)\,=\,0$ for all $i\,\geqslant\, 2$.
\end{theorem}

\begin{proof}
We have already observed above that $R^1\Phi_*ad(\mc A)$
is supported on $\Sigma$. Thus, the map 
$q_1^*T_C\longrightarrow R^1\Phi_*ad(\mc A)$ 
factors through $q_1^*T_C\big\vert_{\Sigma}\longrightarrow\, R^1\Phi_*ad(\mc A)$. 

Let $(c,D)\in \Sigma$ be a point. Then
$c\,\in\, D$. We fix an ordering of the points of $D$ as mentioned at the
beginning of Section \ref{section cohomology of some sheaves}. 
We also choose this ordering in such a way that $c_d\,=\,c$. 
Associated to this ordering, we have the space $S_d$
as constructed in section \ref{section canonical bundle S_D}.
Recall, from \eqref{f7}, the map $g_d:S_d\longrightarrow \mc Q_D$.
We used the same notation to denote the composite map 
$S_d\longrightarrow \mc Q_D\longrightarrow \mc Q$.
Consider the composite
$$
T_{C,c}\,\,=\,\,\bq_1^*T_C\big\vert_{(c,D)}\,\,\longrightarrow\,\,
 R^1\Phi_*ad(\mc A)\big\vert_{(c,D)}
$$
\begin{equation}\label{t1}
\longrightarrow\, 
H^1\left(\mc Q_D,\,ad\left(\mc A\big\vert_{c\times \mc Q_D}\right)\right) 
\,\stackrel{\sim}{\longrightarrow}\, H^1(S_d,\,ad(\mc A\big\vert_{c\times S_d})).
\end{equation}
The last map, which is induced by $g_d$, is an 
isomorphism by Corollary \ref{cor-cohomology over Q_D and S_D}.

We will prove that the composite in \eqref{t1} is an inclusion. 

Take $v\,\in\, T_{C,c}$. Let
\begin{equation}\label{t3}
\alpha\,\in\, H^1\left(\mc Q_D,\,ad\left(\mc A\big\vert_{c\times \mc Q_D}\right)\right)\ \
\text{ and }\ \ \beta\,\in\, H^1\left(S_d,\,ad\left(\mc A\big\vert_{c\times S_d}\right)\right)
\end{equation}
denote the images of $v$ along maps in \eqref{t1}. We need to show that
$\beta\, \not=\, 0$.

Consider the following diagram in which the square is Cartesian
	\[\begin{tikzcd}
		c\times S_d\ar[r,"g_d"]&c\times \mc Q_D\ar[r]\ar[d]& C\times \mc Q\ar[d,"\Phi"]\\
		&(c,D)\ar[r] & C\times C^{(d)}
	\end{tikzcd}
	\]
It is evident that $\alpha$ in \eqref{t3} is the extension class of the short exact sequence
obtained by 
restricting \eqref{rel Atiyah A on C times Q} to $c\times \mc Q_D$. When we further 
pullback this short exact sequence using the map $g_d$, we get the short exact sequence on 
$S_d$ whose extension class is $\beta$ in \eqref{t3}. Thus, $\beta$ is the extension class of the short 
exact sequence obtained by pulling back \eqref{rel Atiyah A on C times Q} along the top 
horizontal row.

However, the top horizontal row 
$c\times S_d\,\stackrel{g_d}{\longrightarrow}\, c\times \mc Q_D\,\longrightarrow\, C\times \mc Q$
factors as
$$c\times S_d\,\longrightarrow\, C\times S_d\,
\xrightarrow{\,\,{\rm Id}_C\times g_d\,\,}\,C\times \mc Q\,.$$ 
We conclude that $\beta$ corresponds to the 
short exact sequence on $S_d$ obtained by pulling back 
the sequence \eqref{rel Atiyah A on C times Q}
along the map ${\rm Id}_C\times g_d$ and restricting it to $c\times S_d$. 
Note that $\mc A\big\vert_{C\times S_d}\,=\,A_{d}$
(see \eqref{univ-quotient-S_D}). By Corollary \ref{base change relative ad-atiyah},
the pullback of \eqref{rel Atiyah A on C times Q}
along the map ${\rm Id}_C\times g_d$ is the relative adjoint Atiyah
sequence for the bundle $A_d$ on $C\times S_d$, that is,
the exact sequence 
\begin{equation}\label{t2}
0\,\longrightarrow\, ad(A_d)\,\longrightarrow\, at_C(A_d)
\,\longrightarrow\, p_C^*T_C\,\longrightarrow\, 0\,.
\end{equation}
We shall use Proposition \ref{Narasimhan-Ramanan} to show 
that the restriction of \eqref{t2} to $c\times S_d$ is a non-trivial
extension. Set $X\,=\,S_{d-1}$, $V\,=\,A_{d-1}$ on $C\times X$ and 
$c\,=\,c_d$ in Proposition \ref{Narasimhan-Ramanan}. Note that the quotient in
\eqref{eqn-sequence of C times P(E_c)}
is exactly the quotient in \eqref{quotient-1} when we take $j\,=\,d$.
As a result, applying Proposition \ref{Narasimhan-Ramanan},
we get that the infinitesimal deformation map of 
$A_d$ at the point $c_d$ is injective. This means that 
the class of the restriction of 
$$0\,\longrightarrow\, \ms End(A_d)\,\longrightarrow\, At_C(A_d)
\,\longrightarrow\, p_C^*T_C\,\longrightarrow\, 0$$
to $c\times S_d$ is non-zero. As $H^1(S_d,\,\mc O_{S_d})\,=\,0$,
it follows that the class of the restriction of \eqref{t2}
to $c\times S_d$ is non-zero. This proves that $\beta$ in \eqref{t3} is nonzero.
Thus the composite in \eqref{t1} is an inclusion.

{}From Lemma \ref{lemma-cohomology computation end(G_d)} it follows
that $h^1(S_d,\, \ms End(A_d\big\vert_{c\times S_d}))\,=\,1$.
Thus, $h^1(S_d,\, ad(A_d\big\vert_{c\times S_d}))\,=\,1$.
In view of the injectivity of the composite in \eqref{t1}, this
implies that the composite map in \eqref{t1} is actually an isomorphism.
	
Since the composite in \eqref{t1} is an isomorphism, it follows that 
the map 
$$R^1\Phi_*ad(\mc A)\big\vert_{(c,D)}\,\longrightarrow\, 
H^1\left(\mc Q_D,\,ad\left(\mc A\big\vert_{c\times \mc Q_D}\right)\right)$$ 
is surjective. By the base change theorem 
\cite[Chapter 3, Theorem 12.11]{Ha} 
the surjectivity of this map implies that it is in fact an isomorphism. 
Moreover, this also implies that
$$\bq_1^*T_C\big\vert_{(c,D)}\,\,\longrightarrow\,\,
R^1\Phi_*ad(\mc A)\big\vert_{(c,D)}$$
is an isomorphism. As $\Sigma$ is integral we easily conclude that 
$R^1\Phi_*ad(\mc A)$ is a line bundle, and the first statement of
the theorem follows easily.
	
Next we prove the second statement. 
Proceeding as above, it suffices to show that for $(c,\,D)\,\in \,\Sigma$ 
$$H^i\left(S_d,\,\ms End\left(A_d\big\vert_{c\times S_d}\right)\right)\,=\,0$$ for $i\,\geqslant\, 2$.
Now this is the content of Lemma \ref{lemma-cohomology computation end(G_d)}.
This completes the proof of the theorem.
\end{proof}

\begin{corollary}\label{corollary main theorems}\mbox{}
\begin{enumerate}
\item There is the following short exact sequence on $\Sigma$
($\Xi$ is the map in \eqref{preliminary ses e3})
\begin{equation*}
0\,\longrightarrow\, ad(q_1^*E\big\vert_{\Sigma})\,\stackrel{\Xi}{\longrightarrow} \,
\Phi_*\ms Hom(\mc A,\,\mc B)\,\longrightarrow\, q_1^*T_C\big\vert_{\Sigma}\longrightarrow 0\,.
\end{equation*}

\item $R^i\Phi_*\ms Hom(\mc A,\,\mc B)\,=\,0$ for $i\,\geqslant\, 1$.

\item $H^i(\Sigma,\, \Phi_*\ms Hom(\mc A,\,\mc B))\,\stackrel{\sim}{\longrightarrow}
\,H^i(\mc Z,\,\ms Hom(\mc A,\,\mc B))$ for all $i$.
	\end{enumerate}
\end{corollary}

\begin{proof}
Statement (1) follows by combining the short exact sequence in the 
statement of Theorem \ref{preliminiary ses} with the isomorphism in
Theorem \ref{description R1-Phi-End(A)}. Statement (2) follows using the isomorphism
in the statement of Theorem \ref{preliminiary ses} and 
the second assertion in Theorem \ref{description R1-Phi-End(A)}.
Statement (3) follows using the Leray spectral sequence and (2).
\end{proof}

\subsection{The tangent bundle of ${\mc Q}$}

The following proposition describing the tangent bundle of ${\mc Q}$ is standard.

\begin{proposition}\label{T_Q}
The tangent bundle of $\mc Q$ is $$T_{\mc Q}
\,\cong\, p_{2*}(\ms Hom(\mc A,\,\mc B)),$$
where $p_2$ is the projection in \eqref{t5}.
\end{proposition}

\begin{proof}
The proof is same as that of \cite[Theorem 7.1]{Str}.
\end{proof}

Recall that associated to a vector bundle $V$ on $C$ 
there is the Secant bundle on $C^{(d)}$
$${Sec}^d(V)\,:=\,q_{2*}[q_1^*V\big\vert_{\Sigma}]\,.$$

\begin{theorem}\label{main theorem}
The following statements hold:
\begin{enumerate}
\item There is a diagram
\begin{equation}\label{ses statement main theorem}
\begin{tikzcd}
0 \ar[r] & q_1^*ad(E)\big\vert_{\Sigma}\ar[r] \ar[d,equal] & q_1^*at(E)\big\vert_{\Sigma} \ar[r] 
\ar[d] & q_1^* T_C\big\vert_{\Sigma} \ar[r] \ar[d,"\cong"] & 0 \\
0 \ar[r] & q_1^*{ad}(E)\big\vert_{\Sigma} \ar[r,"\Xi"] & \Phi_*{\ms Hom}(\mc A,\,\mc B) \ar[r] & 
R^1\Phi_*ad(\mc A) \ar[r] & 0
\end{tikzcd}
\end{equation}
in which the squares commute up to a minus sign, where $q_1$ and $\Phi$ are
the maps in \eqref{t6} and \eqref{t7} respectively.
The right vertical arrow is the one coming from Theorem \ref{description R1-Phi-End(A)}.
In particular, the middle vertical arrow is an isomorphism.

\item $Sec^d(at(E))\,\,\stackrel{\sim}{\longrightarrow}\,\, \phi_*T_{\mc Q}$.

\item $R^i\phi_*T_{\mc Q}\,=\,0$ for all $i\,>\,0$.
\end{enumerate}
\end{theorem}

\begin{proof}
We will first construct diagram \eqref{ses statement main theorem}.
Denote by $\mc F$ the pushout of the diagram
	\begin{equation}\label{eqn-pushout}
	\begin{tikzcd}
	0 \ar[r] & ad(\mc A) \ar[r] \ar[d] & at_C(\mc A) \\
	& {\ms Hom}(\mc A,\,p_1^*E)/{\mc O}_{C\times {\mc Q}} .
	\end{tikzcd}
	\end{equation}
So using Snake lemma the following diagram is obtained:
	\begin{equation}\label{eqn-atiyah of A and E}
	\begin{tikzcd}
	& 0 \ar[d] & 0 \ar[d] & 
	& \\
	0 \ar[r] & ad(\mc A) \ar[r] \ar[d] & at_C(\mc A) \ar[d] \ar[r] &
	p_1^*T_C \ar[r] \ar[d, equal] & 0 \\
	0 \ar[r] & {\ms Hom}(\mc A,\,p_1^*E)/{\mc O}_{C\times{\mc Q}} \ar[r] \ar[d] & \mc F \ar[r] \ar[d] &
	p_1^* T_C \ar[r] & 0 \\
	& {\ms Hom}(\mc A,\,\mc B) \ar[r,equal] \ar[d] &
{\ms Hom}(\mc A,\,\mc B) \ar[d] &
& \\
& 0 & 0 & &
\end{tikzcd}
\end{equation}
Recall that by Corollary \ref{base change relative ad-atiyah} 
the relative adjoint Atiyah
sequence of $p_1^*E$ on $C\times\mc Q$, where $p_1$ is the projection
in \eqref{t5}, is simply 
the pullback of the relative adjoint Atiyah sequence for $E$. From
Corollary \ref{cor-compatibility of atiyah classes} it follows
that the middle row in \eqref{eqn-atiyah of A and E}
coincides with the pushout of relative adjoint Atiyah sequence of
$p_1^*E$ by the morphism $ad(p_1^*E)\,\longrightarrow\,
{\ms Hom}(A,\,p_1^*E)/\mc O$, that is, we have a diagram
\begin{equation}\label{eqn-atiyah of E}
\begin{tikzcd}
0 \ar[r] & p_1^*ad(E) \ar[r] \ar[d] & p_1^*at(E) \ar[r]
\ar[d] & p_1^* T_C \ar[r] \ar[d,equal] & 0 \\
0 \ar[r] & {\ms Hom}(\mc A,\,p_1^*E)/{\mc O}_{C\times{\mc Q}}\ar[r] & \mc F \ar[r]
& p_1^* T_C \ar[r] & 0 \\
\end{tikzcd}
\end{equation}
Combining (\ref{eqn-atiyah of A and E}) and (\ref{eqn-atiyah of E}) we get maps 
$$p_1^*at(E) \,\longrightarrow\, \mc F\,\longrightarrow\, {\ms Hom}(\mc A,\,
\mc B)\,.$$
Applying $\Phi_*$ we get a map
\begin{equation}\label{eqn-at(E) to Hom(A,B)}
q_1^*at(E)\,\,\longrightarrow\,\, \Phi_*{\ms Hom}(\mc A,\,\mc B)
\end{equation}
Next we show that in the following diagram the squares commute up to a minus sign:
\begin{equation}\label{eqn-commutativity of at(E) and Hom(A,B)}
\begin{tikzcd}
0 \ar[r] & q_1^*ad(E) \ar[r] \ar[d] & q_1^*at(E) \ar[r]
\ar[d] & q_1^* T_C \ar[r] \ar[d] & 0 \\
0 \ar[r] & q_1^*{ad}(E)\big\vert_{\Sigma} \ar[r] & \Phi_*{\ms Hom}(\mc A,\,\mc B) \ar[r] & 
R^1\Phi_*ad(\mc A) \ar[r] & 0
\end{tikzcd}
\end{equation}
Here the bottom sequence is the one in 
Theorem \ref{preliminiary ses}. The middle vertical arrow is given by
\eqref{eqn-at(E) to Hom(A,B)} while the right vertical arrow is given by the 
boundary map of the sequence obtained by applying $\Phi_*$ to the 
relative adjoint Atiyah sequence of $\mc A$. 
The commutativity of the box in the left is evident as $\Phi_*{\ms Hom}(\mc A,\,\mc B)$
is supported on $\Sigma$.
For the commutativity of the box in the right, first recall that since 
$\mc F$ is a pushout of the diagram (\ref{eqn-pushout}) we have a diagram
in which the squares commute up to a minus sign
\[
\begin{tikzcd}
0 \ar[r] & ad(\mc A) \ar[r] & at_C(\mc A) \ar[r] &
p_1^* T_C \ar[r] & 0 \\
0 \ar[r] & ad(\mc A) \ar[r] \ar[u,equal] \ar[d,equal] & at_C(\mc A)\oplus 
{\ms Hom}(\mc A,\,p_1^*E)/{\mc O}_{C\times{\mc Q}} \ar[r] \ar[u] \ar[d] &
\mc F \ar[r] \ar[d] \ar[u] & 0 \\
0 \ar[r] & ad(\mc A) \ar[r] & {\ms Hom}(\mc A,\,p_1^*E)/{\mc O}_{C\times{\mc Q}} \ar[r] & 
{\ms Hom}(\mc A,\,\mc B) \ar[r] & 0
\end{tikzcd}
\]
Applying $\Phi_*$ we get a diagram 
\[
\begin{tikzcd}
&q_1^*T_C \ar[r] & R^1\Phi_*ad(\mc A) \ar[d,equal]\\
q_1^*at(E)\ar[r]&\Phi_*\mc F \ar[r] \ar[u] \ar[d] & R^1\Phi_*ad(\mc A) \ar[d,equal] \\
&\Phi_*{\ms Hom}(\mc A,\,\mc B) \ar[r] & R^1\Phi_*ad(\mc A)
\end{tikzcd}
\]
in which the squares commute up to a minus sign.
This shows the commutativity of the right square in (\ref{eqn-commutativity of at(E) and Hom(A,B)})
up to a minus sign.
The first assertion of the theorem is proved because all sheaves in 
the lower row of \eqref{ses statement main theorem} are supported on $\Sigma$.

It follows that we have an isomorphism 
$q^*_1 at(E)\big\vert_{\Sigma}\,\cong\, \Phi_*{\ms Hom}(\mc A,\,\mc B)$.
Applying $q_{2*}$ to this and using Proposition \ref{T_Q}
yields the second assertion.
The third assertion is deduced from 
Corollary \ref{corollary main theorems} using $q_2\big\vert_{\Sigma}$ 
and $p_2\big\vert_{\mc Z}$ are finite maps.
\end{proof}

\subsection{Computation of cohomologies of $T_{\mc Q}$}

The following theorem is deduced using the fact that
the middle vertical arrow in \eqref{ses statement main theorem}
is an isomorphism.

\begin{theorem}\label{cor-cohomology of T_Q}\mbox{}
\begin{enumerate}
\item Let $g_C$ be the genus of $C$. For all $d-1\, \geqslant\, i\, \geqslant\, 0$,
\begin{align*}
H^i(\mc Q,\,T_{\mc Q})\,=\,H^0(C,\,at(E))\otimes & \bigwedge^ i H^1(C,\,\mc O_{C})\\
& \bigoplus H^1(C,\,at(E))\otimes 
\bigwedge^{i-1}H^{1}(C,\,\mc O_{C})\, .
\end{align*}
In particular,
$$h^i(\mc Q,\, T_{\mc Q})\,=\, {g_C \choose i} \cdot h^0(C,\,at(E))+
{g_C \choose {i-1}} \cdot h^1(C,\, at(E))\,.$$

\item 
When $i\,=\,d$,
$$H^d(\mc Q, \, T_{\mc Q})\, =\, \bigwedge^{d-1}H^1(C,\,\mc O_C) \otimes h^1(C, \, at(E))\,.$$
In particular,
$$h^d(\mc Q, \, T_{\mc Q})\, =\, {g_C \choose {d-1}} \cdot h^1(C, \, at(E))\,.$$

\item For all $i\,\geq\, d+1$, $$H^i(\mc Q,\, T_{\mc Q})\,=\,0.$$
\end{enumerate}
\end{theorem}

\begin{proof}
From Corollary \ref{corollary main theorems}, Proposition \ref{T_Q} and
Theorem \ref{main theorem} it follows that 
$$H^i(\mc Q,\, T_{\mc Q})\,\cong\, H^i(\mc Z,\,\ms Hom(\mc A,\,\mc B))
\,\cong\, H^i\left(\Sigma,\,q_1^*at(E)\big\vert_{\Sigma}\right)\,.$$
But $\Sigma\,\cong\, C\times C^{(d-1)}$, and $q_1\big\vert_{\Sigma}\,:\,\Sigma
\,\longrightarrow\, C$ is just the first projection \cite[Section 10, Chapter 11]{ACGH2}. 
Therefore by K\"unneth formula we get that 
\begin{align*}
H^i(\mc Q,\,T_{\mc Q})\,=\,H^0(C,\,at(E))\otimes & H^i\left(C^{(d-1)},\,
\mc O_{C^{(d)}}\right)\\
& \bigoplus H^1(C,\,at(E))\otimes 
H^{i-1}(C^{(d-1)},\, \mc O_{C^{(d)}})\,.
\end{align*}
Now by \cite[Equation 11.1]{Macdonald} we have
$$H^i(C^{(d-1)},\, \mc O_{C^{(d-1)}})\,=\bigwedge^iH^1(C,\mc O_C)$$ if 
$0 \, \leq \, i \, \leqslant \, d-1$ and $0$ otherwise. 
The statement of the theorem now follows immediately. 
\end{proof}

\begin{theorem}\label{cor-cohomology of T_Q-1}\mbox{}
\begin{enumerate}
\item If $d,g_C\,\geqslant\, 2$, then
$$h^1(\mc Q,\,T_{\mc Q})\,=\,g_C\cdot h^0(C,ad(E))+ h^1(C,ad(E))+3g_C-3\,.$$ 
In particular, the dimension of the space $H^1(\mc Q,\,T_{\mc Q})$ depends only
on $C$ and $E$ and is independent of $d$.

\item Let $d,\,g_C\,\geqslant\, 2$. Then there is an exact sequence
$$
0 \,\longrightarrow\, H^1\left(\Sigma,\,q_1^*ad(E)\big\vert_{\Sigma}\right)\,\longrightarrow\,
H^1(\mc Q,\,T_{\mc Q})\,\longrightarrow\, H^1(C,\,T_{C})\,\longrightarrow\, 0\,,
$$
where $H^1\left(\Sigma,\,q_1^*ad(E)\big\vert_{\Sigma}\right)\,\cong\, H^0(C,\,ad(E))\otimes H^1(C,
\,\mc O_C) 
\bigoplus H^1(C,\,ad(E))$.

\end{enumerate}
\end{theorem}

\begin{proof}
Recall that we have an exact sequence on $C$
\begin{equation}\label{t8}
0 \,\longrightarrow\, ad(E)\,\longrightarrow\, at(E)
\,\longrightarrow\, T_C\,\longrightarrow\, 0\,.
\end{equation}
Since $g_C\, \geqslant\, 2$, it follows that $H^0(C,\,T_C)\,=\,0$. Therefore,
the long exact sequence of cohomologies for \eqref{t8} gives
$H^0(C,\,ad(E))\,=\,H^0(C,\,at(E))$ together with an exact sequence
$$0 \,\longrightarrow \,H^1(C,\,ad(E))\,\longrightarrow\, H^1(C,\,at(E))
\,\longrightarrow\, H^1(C,\,T_C)\,\longrightarrow\, 0\,.$$
This shows that $h^1(C,\,at(E))\,=\,h^1(C,ad(E))+3g_C-3$. The first statement
now follows from Theorem \ref{cor-cohomology of T_Q}.

We will prove the second statement. Using K\"unneth formula it follows that
$$H^1(\Sigma,\, q_1^*T_C\big\vert_{\Sigma})\,\cong\, H^1(C,\,T_C).$$
Then using \eqref{ses statement main theorem}, we get an exact sequence 
\begin{equation}\label{t9}
H^1\left(\Sigma,\,q_1^*ad(E)\big\vert_{\Sigma})
\,\longrightarrow\, H^1(\mc Q,\,T_{\mc Q}\right)\,\longrightarrow\, H^1(C,\,T_{C})\,.
\end{equation}
It was shown in (1) that the dimension of the middle term
equals the sum of the dimensions of the extreme terms. It follows
that the sequence in \eqref{t9} must be injective on the left and surjective 
on the right. The last isomorphism follows using K\"unneth formula
and the isomorphism $H^1(C^{(d-1)},\,\mc O_{C^{(d-1)}})\,\cong\,H^1(C,\,\mc O_C)$.
\end{proof}

In \cite{BDH-aut} it was shown that when $E\,\cong \,\mc O^r_C$ and $g_C\,\geqslant\, 2$,
then
$$H^0(\mc Q,\,T_{\mc Q})\,=\,\mf{sl}(r,\mb C)\,=\,H^0(C,ad(\mc O^r_C))\,.$$
This was generalized in \cite{G19} where it was shown that when $g_C\,\geqslant\, 2$ and 
either $E$ is semistable or ${\rm rank}(E)\,\geqslant\, 3$, then
$$H^0(\mc Q,\,T_{\mc Q})\,=\,H^0(C,\,ad(E))\,.$$
It follows immediately from Theorem \ref{cor-cohomology of T_Q} that we can drop these 
assumptions on $E$ to get the following general statement.

\begin{corollary}
Let $g_C\,\geqslant \,2$. Then 
$$H^0(\mc Q,\,T_{\mc Q})\,=\,H^0(C,\,{at}(E))\,=\,H^0(C,\,ad(E))\,.$$
\end{corollary}

\begin{corollary}\label{cor-direct image of End(A)}
There is an isomorphism of sheaves 
$$\Phi_*\ms End(\mc A)\,\,\cong\,\, \mc O_{C\times C^{(d)}}\bigoplus ad(q_1^*E)(-\Sigma)\,.$$
\end{corollary}

\begin{proof}
Consider the exact sequence in \eqref{u1}. In it, from
Proposition \ref{propn-direct image of A} we have
$$\Phi_*{\ms Hom}(\mc A,\,p_1^*E)\,\cong\, {\ms End}(q_1^*E).$$ Also, from
\eqref{preliminary ses e3} we know that
the image of the map ${\ms End}(q_1^*E)
\,\longrightarrow \,\Phi_*{\ms Hom}(\mc A,\,\mc B)$ is $ad(q_1^*E)\big\vert_{\Sigma}$.
Consequently, \eqref{u1} produces an exact sequence
$$0 \,\longrightarrow\, \Phi_*{\ms End}(\mc A)\,\longrightarrow\, {\ms End}(q_1^*E)
\,\longrightarrow\, ad(q_1^*E)\big\vert_{\Sigma}\,\longrightarrow\, 0\,.$$
Writing ${\ms End}(q_1^*E)\,=\,\mc O_{C\times C^{(d)}}\bigoplus ad(q_1^*E)$,
the result follows easily.
\end{proof}

\begin{lemma}\label{lemma-direct image of ideal sheaf}
The vanishing statements $$q_{1*}\mc O(-\Sigma)\,=\,R^1q_{1*}\mc O(-\Sigma)\,=\,0$$ hold.
\end{lemma}

\begin{proof}
Consider the exact sequence
\begin{equation}\label{u2}
0 \,\longrightarrow\, \mc O(-\Sigma)\,\longrightarrow\, 
\mc O_{C\times C^{(d)}}\,\longrightarrow\, \mc O_{\Sigma}\,\longrightarrow\, 0\,.
\end{equation}
Note that the map $q_{1*}\mc O_{C\times C^{(d)}}\,\longrightarrow\, q_{1*}\mc O_{\Sigma}$ 
is an isomorphism since both of these sheaves are 
isomorphic to $\mc O_C$. Therefore we have $q_{1*}\mc O(-\Sigma)\,=\,0$ and an
exact sequence
\begin{equation}\label{u3}
0 \,\longrightarrow\, R^1q_{1*}\mc O(-\Sigma)\,\longrightarrow\, 
R^1q_{1*}\mc O_{C\times C^{(d)}}\,\stackrel{\varpi}{\longrightarrow}\,
R^1q_{1*}\mc O_{\Sigma}
\end{equation}
from \eqref{u2}. In view of it, to prove the lemma it suffices to show that
$\varpi$ in \eqref{u3} is an isomorphism.

The sheaves $\mc O_{C\times C^{(d)}}$ and $\mc O_{\Sigma}$
are flat over $C$ and if $c\,\in \,C$, the induced map 
$$\mc O_{C\times C^{(d)}}\big\vert_{c\times C^{(d)}}\,\longrightarrow\,
\mc O_{\Sigma}\big\vert_{c\times C^{(d)}}$$
coincides with the homomorphism
$\mc O_{C^{(d)}}\,\longrightarrow \,\mc O_{C^{(d-1)}}$
corresponding to the inclusion map $C^{(d-1)}\,\hookrightarrow\, C^{(d)}$
defined by $D\,\longmapsto\, D+c$.
By \cite[Corollary 1.5]{Kempf} and the remark following it we 
know that the induced map 
$$H^1\left(C^{(d)},\,\mc O_{C^{(d)}}\right)\,\longrightarrow\,
H^1\left(C^{(d)},\,\mc O_{C^{(d-1)}}\right)$$
is an isomorphism.
Using Grauert's theorem this implies that $\varpi$ in \eqref{u3} is an
isomorphism. This completes the proof of the lemma.
\end{proof}

In \cite{G18} it was shown that when $E\,\cong\, \mc O^n_C$ 
for some $n\,\geqslant\, 1$, the vector bundle $\mc A$ is stable 
with respect to certain polarizations on $\mc Q$. 
In particular, $H^0(C\times \mc Q,\,\ms End(\mc A))\,=\,1$ in that case. 
In the following corollary we see that this is in 
fact true in general without any assumptions on $E$.

\begin{corollary}
The equality $h^0(C\times \mc Q,\, {\ms End}(\mc A))\,=\,1$ holds.
\end{corollary}

\begin{proof}
Combining Corollary \ref{cor-direct image of End(A)} and 
Lemma \ref{lemma-direct image of ideal sheaf} we have
$$q_{1*}\Phi_*\ms End(\mc A) \,=\, q_{1*}\left[\mc O_{C\times C^{(d)}}
\bigoplus q_1^*ad(E)(-\Sigma)\right]\, =\, \mc O_C \,.$$
The corollary now follows immediately.
\end{proof}

\begin{corollary}\label{c916}
Let $g_C\,\geqslant\, 2$. Then
$$H^1(C\times \mc Q,\,{\ms End}(\mc A))
\,=\,H^1(C\times \mc Q,\,\mc O_{C\times \mc Q})\,=\,
H^1(C,\,\mc O_C)\oplus H^1(\mc Q,\,\mc O_{\mc Q})\,.$$
\end{corollary}

\begin{proof}
{}From the Leray Spectral sequence, it follows that there is an exact sequence
\begin{align*}
0 \,\longrightarrow\, H^1(C\times C^{(d)},\,\Phi_*{\ms End}(\mc A))
&\,\longrightarrow\, H^1(C\times \mc Q,\,\ms End(\mc A))\\
&\longrightarrow\, H^0(C\times C^{(d)},\,R^1\Phi_*{\ms End}(\mc A))\,.
\end{align*}
{}From Theorem \ref{description R1-Phi-End(A)}
it follows that $R^1\Phi_*\ms End(\mc A)\,\cong\, q_1^*T_C\big\vert_{\Sigma}$. 
Therefore, we have 
$$H^0(C\times C^{(d)},\,R^1\Phi_*\ms End(\mc A))
\,=\,H^0\left(C\times C^{(d)},\, q_1^*T_C\big\vert_{\Sigma}\right)\,=\,0\,.$$
Here the last equality follows from the assumption that $g_C\,\geqslant\, 2$.
So it suffices to compute $H^1(C\times C^{(d)},\,\Phi_*{\ms End}(\mc A))$.

By Lemma \ref{lemma-direct image of ideal sheaf} we have
$$q_{1*}(q_{1}^*ad(E)(-\Sigma))\,=\,0\,=\,R^1q_{1*}(q_{1}^*ad(E)(-\Sigma))\,.$$
Using this and the Leray Spectral sequence it follows that
$$H^1(C\times C^{(d)},\,q_1^*ad(E)(-\Sigma))\,=\,0\,.$$
Therefore, using Corollary \ref{cor-direct image of End(A)} it is deduced that
\begin{align*}
H^1(C\times C^{(d)},\,\Phi_*{\ms End}(\mc A))
&\,=\,H^1\left(C\times C^{(d)},\,\mc O_{C\times C^{(d)}}\right)\\
&=\,H^1\left(C,\,\mc O_C)\oplus H^1(C^{(d)},\,\mc O_{C^{(d)}}\right)\\
&=\,H^1(C,\,\mc O_C)\oplus H^1(\mc Q,\,\mc O_{\mc Q})\,.
\end{align*}
Here the last equality follows from Corollary \ref{cor-direct image of structure sheaf}
and the Leray spectral sequence. This completes the proof of the Corollary.
\end{proof}


\begin{thebibliography}{BDW96}
	\expandafter\ifx\csname url\endcsname\relax
	\def\url#1{\texttt{#1}}\fi
	\expandafter\ifx\csname doi\endcsname\relax
	\def\doi#1{\burlalt{doi:#1}{http://dx.doi.org/#1}}\fi
	\expandafter\ifx\csname urlprefix\endcsname\relax\def\urlprefix{URL }\fi
	\expandafter\ifx\csname href\endcsname\relax
	\def\href#1#2{#2}\fi
	\expandafter\ifx\csname burlalt\endcsname\relax
	\def\burlalt#1#2{\href{#2}{#1}}\fi
	
	\bibitem[ACG11]{ACGH2}
	Enrico Arbarello, Maurizio Cornalba, and Phillip~A. Griffiths.
	\newblock {\em Geometry of algebraic curves. {V}olume {II}}, volume 268 of {\em
		Grundlehren der mathematischen Wissenschaften [Fundamental Principles of
		Mathematical Sciences]}.
	\newblock Springer, Heidelberg, 2011.
	\newblock \doi{10.1007/978-3-540-69392-5}.
	\newblock With a contribution by Joseph Daniel Harris.
	
	\bibitem[Ati57]{Atiyah}
	M.~F. Atiyah.
	\newblock Complex analytic connections in fibre bundles.
	\newblock {\em Trans. Amer. Math. Soc.}, 85:181--207, 1957.
	\newblock \doi{10.2307/1992969}.
	
	\bibitem[BDH15]{BDH-aut}
	Indranil Biswas, Ajneet Dhillon, and Jacques Hurtubise.
	\newblock Automorphisms of the quot schemes associated to compact {R}iemann
	surfaces.
	\newblock {\em Int. Math. Res. Not. IMRN}, (6):1445--1460, 2015.
	\newblock \doi{10.1093/imrn/rnt259}.
	
	\bibitem[BDW96]{BDW}
	Aaron Bertram, Georgios Daskalopoulos, and Richard Wentworth.
	\newblock Gromov invariants for holomorphic maps from {R}iemann surfaces to
	{G}rassmannians.
	\newblock {\em J. Amer. Math. Soc.}, 9(2):529--571, 1996.
	\newblock \doi{10.1090/S0894-0347-96-00190-7}.
	
	\bibitem[BFP20]{BFP}
	Massimo Bagnarol, Barbara Fantechi, and Fabio Perroni.
	\newblock On the motive of {Q}uot schemes of zero-dimensional quotients on a
	curve.
	\newblock {\em New York J. Math.}, 26:138--148, 2020.
	
	\bibitem[BGL94]{BGL}
	Emili Bifet, Franco Ghione, and Maurizio Letizia.
	\newblock On the {A}bel-{J}acobi map for divisors of higher rank on a curve.
	\newblock {\em Math. Ann.}, 299(4):641--672, 1994.
	\newblock \doi{10.1007/BF01459804}.
	
	\bibitem[BH93]{Bruns-Herzog}
	Winfried Bruns and J\"{u}rgen Herzog.
	\newblock {\em Cohen-{M}acaulay rings}, volume~39 of {\em Cambridge Studies in
		Advanced Mathematics}.
	\newblock Cambridge University Press, Cambridge, 1993.
	
	\bibitem[BL11]{biswas-laytimi}
	Indranil Biswas and Fatima Laytimi.
	\newblock Direct image and parabolic structure on symmetric product of curves.
	\newblock {\em J. Geom. Phys.}, 61(4):773--780, 2011.
	\newblock \doi{10.1016/j.geomphys.2010.12.005}.
	
	\bibitem[BR08]{Biswas-Raghavendra}
	Indranil Biswas and N.~Raghavendra.
	\newblock The {A}tiyah-{W}eil criterion for holomorphic connections.
	\newblock {\em Indian J. Pure Appl. Math.}, 39(1):3--47, 2008.
	
	\bibitem[BRa13]{BR}
	Indranil Biswas and Nuno~M. Rom\~{a}o.
	\newblock Moduli of vortices and {G}rassmann manifolds.
	\newblock {\em Comm. Math. Phys.}, 320(1):1--20, 2013.
	\newblock \doi{10.1007/s00220-013-1704-3}.
	
	\bibitem[Che12]{Chen}
	Ting Chen.
	\newblock The associated map of the nonabelian {G}auss-{M}anin connection.
	\newblock {\em Cent. Eur. J. Math.}, 10(4):1407--1421, 2012.
	\newblock \doi{10.2478/s11533-011-0110-3}.
	
	\bibitem[Fan94]{Fantechi}
	Barbara Fantechi.
	\newblock Deformations of symmetric products of curves.
	\newblock In {\em Classification of algebraic varieties ({L}'{A}quila, 1992)},
	volume 162 of {\em Contemp. Math.}, pages 135--141. Amer. Math. Soc.,
	Providence, RI, 1994.
	\newblock \doi{10.1090/conm/162/01531}.
	
	\bibitem[Gan18]{G18}
	Chandranandan Gangopadhyay.
	\newblock Stability of sheaves over {Q}uot schemes.
	\newblock {\em Bull. Sci. Math.}, 149:66--85, 2018.
	\newblock \doi{10.1016/j.bulsci.2018.08.001}.
	
	\bibitem[Gan19]{G19}
	Chandranandan Gangopadhyay.
	\newblock Automorphisms of relative {Q}uot schemes.
	\newblock {\em Proc. Indian Acad. Sci. Math. Sci.}, 129(5):Paper No. 85, 11,
	2019, \burlalt{1812.10446}{http://arxiv.org/abs/1812.10446}.
	\newblock \doi{10.1007/s12044-019-0522-8}.
	
	\bibitem[GS20]{GS}
	Chandranandan Gangopadhyay and Ronnie Sebastian.
	\newblock Fundamental group schemes of some {Q}uot schemes on a smooth
	projective curve.
	\newblock {\em J. Algebra}, 562:290--305, 2020.
	\newblock \doi{10.1016/j.jalgebra.2020.06.025}.
	
	\bibitem[GS21]{GS-nef}
	Chandranandan Gangopadhyay and Ronnie Sebastian.
	\newblock Nef cones of some {Q}uot schemes on a smooth projective curve.
	\newblock {\em C. R. Math. Acad. Sci. Paris}, 359:999--1022, 2021.
	\newblock \doi{10.5802/crmath.245}.
	
	\bibitem[Har77]{Ha}
	Robin Hartshorne.
	\newblock {\em Algebraic geometry}.
	\newblock Springer-Verlag, New York-Heidelberg, 1977.
	\newblock Graduate Texts in Mathematics, No. 52.
	
	\bibitem[HL10]{HL}
	Daniel Huybrechts and Manfred Lehn.
	\newblock {\em The geometry of moduli spaces of sheaves}.
	\newblock Cambridge Mathematical Library. Cambridge University Press,
	Cambridge, second edition, 2010.
	\newblock \doi{10.1017/CBO9780511711985}.
	
	\bibitem[HPL21]{HP}
	Victoria Hoskins and Simon Pepin~Lehalleur.
	\newblock On the {V}oevodsky motive of the moduli stack of vector bundles on a
	curve.
	\newblock {\em Q. J. Math.}, 72(1-2):71--114, 2021.
	\newblock \doi{10.1093/qmathj/haaa023}.
	
	\bibitem[Kem81]{Kempf}
	George~R. Kempf.
	\newblock Deformations of symmetric products.
	\newblock In {\em Riemann surfaces and related topics: {P}roceedings of the
		1978 {S}tony {B}rook {C}onference ({S}tate {U}niv. {N}ew {Y}ork, {S}tony
		{B}rook, {N}.{Y}., 1978)}, volume~97 of {\em Ann. of Math. Stud.}, pages
	319--341. Princeton Univ. Press, Princeton, N.J., 1981.
	
	\bibitem[Laz04]{Laz}
	Robert Lazarsfeld.
	\newblock {\em Positivity in algebraic geometry. {I}}, volume~48 of {\em
		Ergebnisse der Mathematik und ihrer Grenzgebiete. 3. Folge. A Series of
		Modern Surveys in Mathematics [Results in Mathematics and Related Areas. 3rd
		Series. A Series of Modern Surveys in Mathematics]}.
	\newblock Springer-Verlag, Berlin, 2004.
	\newblock \doi{10.1007/978-3-642-18808-4}.
	\newblock Classical setting: line bundles and linear series.
	
	\bibitem[Mac62]{Macdonald}
	I.~G. Macdonald.
	\newblock Symmetric products of an algebraic curve.
	\newblock {\em Topology}, 1:319--343, 1962.
	\newblock \doi{10.1016/0040-9383(62)90019-8}.
	
	\bibitem[Mat65]{Mattuck}
	Arthur Mattuck.
	\newblock Secant bundles on symmetric products.
	\newblock {\em Amer. J. Math.}, 87:779--797, 1965.
	\newblock \doi{10.2307/2373245}.
	
	\bibitem[Mat86]{Mat}
	Hideyuki Matsumura.
	\newblock {\em Commutative ring theory}, volume~8 of {\em Cambridge Studies in
		Advanced Mathematics}.
	\newblock Cambridge University Press, Cambridge, 1986.
	\newblock Translated from the Japanese by M. Reid.
	
	\bibitem[Nit09]{Nitsure}
	Nitin Nitsure.
	\newblock Deformation theory for vector bundles.
	\newblock In {\em Moduli spaces and vector bundles}, volume 359 of {\em London
		Math. Soc. Lecture Note Ser.}, pages 128--164. Cambridge Univ. Press,
	Cambridge, 2009.
	
	\bibitem[NR75]{NR}
	M.~S. Narasimhan and S.~Ramanan.
	\newblock Deformations of the moduli space of vector bundles over an algebraic
	curve.
	\newblock {\em Ann. of Math. (2)}, 101:391--417, 1975.
	\newblock \doi{10.2307/1970933}.
	
	\bibitem[OP21]{OP}
	Dragos Oprea and Rahul Pandharipande.
	\newblock Quot schemes of curves and surfaces: virtual classes, integrals,
	{E}uler characteristics.
	\newblock {\em Geom. Topol.}, 25(7):3425--3505, 2021.
	\newblock \doi{10.2140/gt.2021.25.3425}.
	
	\bibitem[Ric20]{Ricolfi}
	Andrea~T. Ricolfi.
	\newblock Virtual classes and virtual motives of {Q}uot schemes on threefolds.
	\newblock {\em Adv. Math.}, 369:107182, 32, 2020.
	\newblock \doi{10.1016/j.aim.2020.107182}.
	
	\bibitem[Ser06]{Sern}
	Edoardo Sernesi.
	\newblock {\em Deformations of algebraic schemes}, volume 334 of {\em
		Grundlehren der Mathematischen Wissenschaften [Fundamental Principles of
		Mathematical Sciences]}.
	\newblock Springer-Verlag, Berlin, 2006.
	
	\bibitem[Stk]{Stk}
	{The Stack Project}.
	\newblock \url{https://stacks.math.columbia.edu}.
	
	\bibitem[Str87]{Str}
	Stein~Arild Stromme.
	\newblock On parametrized rational curves in {G}rassmann varieties.
	\newblock In {\em Space curves ({R}occa di {P}apa, 1985)}, volume 1266 of {\em
		Lecture Notes in Math.}, pages 251--272. Springer, Berlin, 1987.
	\newblock \doi{10.1007/BFb0078187}.
	
\end{thebibliography}

\end{document}